\documentclass[a4paper,10pt]{amsart}

\usepackage{hyperref}

\usepackage{verbatim,exscale,amssymb,amsthm,mathrsfs,amsrefs}

\numberwithin{equation}{section}

\usepackage[all]{xy}

\xyoption{ps}
\SelectTips{cm}{}

\newdir{ >}{{}*!/-5pt/\dir{>}}

\DeclareMathOperator{\supph}{ssupp}

\newcommand{\eg}{\textrm{e.g.}}
\newcommand{\ie}{\textrm{i.e.}}
\newcommand{\cf}{\textrm{cf.}}

\newcommand{\tot}{\mathrm{tot}}
\newcommand{\id}{\mathrm{id}}

\newcommand{\loc}{\mathrm{loc}}

\newcommand{\Id}{\mathop{\mathrm{Id}}}
\newcommand{\spec}{\mathop{\mathrm{Spec}}}
\newcommand{\specgr}{\mathop{\mathrm{Spec^h}}}
\newcommand{\Spec}{\mathop{\mathrm{Spec}}}
\newcommand{\spc}{\mathop{\mathrm{Spc}}}
\newcommand{\Spc}{\mathop{\mathrm{Spc}}}
\newcommand{\Proj}{\mathop{\mathrm{Proj}}}
\newcommand{\Pic}{\mathop{\mathrm{Pic}}}

\DeclareMathOperator{\supp}{supp}

\newcommand{\Hom}{\mathrm{Hom}}

\newcommand{\Mor}{\mathop{\mathrm{Mor}}}

\newcommand{\End}{\mathrm{End}}

\newcommand{\Ann}{\mathrm{Ann}}

\newcommand{\op}{\mathrm{op}}

\newcommand{\cone}{\mathop{\mathrm{cone}}}

\newcommand{\colim}{\mathop{\mathrm{colim}}}

\newcommand{\unit}{\mathbf{1}} 
\newcommand{\obj}{\mathop{\mathrm{obj}}}

\newcommand{\Ker}{\mathop{\mathsf{Ker}}}

\newcommand{\perf}{\mathrm{perf}}

\newcommand{\Gr}{\mathrm{Gr}}

\newcommand{\Mod}{\mathrm{Mod}}
\newcommand{\GrMod}{\mathop{\mathrm{GrMod}}}

\newcommand{\D}{\mathrm{D}} 

\newcommand{\Ab}{\mathrm{Ab}}

\newcommand{\Z}{\mathbb{Z}}

\newcommand{\frp}{\mathfrak{p}}

\newcommand{\LL}{\mathbf{L}}

\newtheorem{thm}[equation]{Theorem}
\newtheorem{lemma}[equation]{Lemma}
\newtheorem{thm-defi}[equation]{Theorem-Definition}
\newtheorem{prop}[equation]{Proposition}
\newtheorem{cor}[equation]{Corollary}

\newtheorem*{thm*}{Theorem}
\newtheorem*{lemma*}{Lemma}
\newtheorem*{cor*}{Corollary}
\newtheorem*{conj*}{Conjecture}
\newtheorem*{question*}{Question}

\theoremstyle{definition}

\newtheorem{defi}[equation]{Definition}

\newtheorem*{conv*}{Conventions}

\newtheorem*{Related work}{Related work}

\theoremstyle{remark}

\newtheorem{notation}[equation]{Notation}
\newtheorem{remark}[equation]{Remark}
\newtheorem*{remark*}{Remark}
\newtheorem{remarks}[equation]{Remarks}
\newtheorem{example}[equation]{Example}
\newtheorem{examples}[equation]{Examples}

\begin{document}
\title[Even more spectra]{Even more spectra: tensor triangular comparison maps via graded commutative 2-rings}

\author{Ivo Dell'Ambrogio}
\address{Universit\"at Bielefeld, Fakult\"at f\"ur Mathematik, BIREP Gruppe, Postfach 10\,01\,31, 33501 Bielefeld, Germany.}
\email{ambrogio@math.uni-bielefeld.de}

\author{Greg Stevenson}
\address{Universit\"at Bielefeld, Fakult\"at f\"ur Mathematik, BIREP Gruppe, Postfach 10\,01\,31, 33501 Bielefeld, Germany.}
\email{gstevens@math.uni-bielefeld.de}



\begin{abstract}
We initiate the theory of graded commutative 2-rings, a categorification of graded commutative rings. The goal is to provide a systematic generalization of Paul Balmer's comparison maps between the spectrum of tensor-triangulated categories and the Zariski spectra of their central rings. By applying our constructions, we compute the spectrum of the derived category of perfect complexes over any graded commutative ring, and we associate to every scheme with an ample family of line bundles an embedding into the spectrum of an associated graded commutative 2-ring. 
\end{abstract}

 \maketitle
 \date{}

\tableofcontents

\section{Introduction}
\label{intro}

\subsection{Motivation}
Consider a tensor-triangulated category $\mathcal T$, that is, an essentially small triangulated category~$\mathcal T$ equipped with a bi-exact symmetric tensor product. Paul Balmer \cite{balmer:prime} associates to $\mathcal T$ a functorial invariant -- a topological space called the \emph{spectrum} of~$\mathcal T$ and denoted by $\Spc \mathcal T$ -- which turns out to be the starting point of a powerful geometric theory known as \emph{tensor triangular geometry}. We refer to Balmer's 2011 ICM address \cite{balmer:icm} for an account of this theory and its many applications.

In order to successfully apply the abstract theory to examples it is essential that one provides a relevant description of the spectrum $\Spc \mathcal T$. This is in general a difficult task, not least because such a computation is equivalent to providing a classification of the thick tensor-ideals of $\mathcal T$. In \cite{balmer:spec3}, Balmer has introduced, in full generality, a natural and continuous \emph{comparison map}
\[
\rho_\mathcal T \colon \Spc (\mathcal T) \to \Spec( \End_\mathcal T(\unit))
\]
from the spectrum of $\mathcal T$ to the Zariski spectrum of the endomorphism ring of~$\unit$, the tensor unit. There are useful criteria to check that this map is surjective, which is often the case, and therefore $\rho_\mathcal{T}$ displays $\Spc \mathcal T$ as being fibered over a more familiar and tractable space. On the other hand, injectivity is more subtle and seems to occur less frequently in examples. 
This can be remedied somewhat by considering \emph{graded} endomorphism rings of~$\unit$: if~$g$ is a tensor-invertible object of~$\mathcal T$, one can define a $\Z$-graded ring $\End^{*,g}_\mathcal T(\unit):= \bigoplus_{n\in\Z} \Hom_\mathcal T(\unit , g^{\otimes n})$ that is graded commutative by the Eckmann-Hilton argument. One obtains in this way a graded version
\[
\rho_\mathcal T^{*,g}\colon \Spc(\mathcal T) \to \specgr(\End^{*,g}_\mathcal T(\unit))
\]
 of the comparison map, where the space on the right-hand side is now the spectrum of homogeneous prime ideals. 
Then $\rho_\mathcal T^{*,g}$ is injective -- an embedding -- if we take $\mathcal T$ to be $\D^\perf(X)$ for a  projective variety~$X$, with $g=\mathcal O(1)$ (see \cite{balmer:spec3}*{Remark 8.2}).
If we identify $X=\Proj(\End^{*,g}_\mathcal T(\unit))$, then $\rho_\mathcal T^{*,g}$ is a homeomorphism of  $\Spc(\mathcal T)$ onto~$X$.

It is now tempting to try and produce an endomorphism ring of~$\unit$ which is graded over, say, the Picard group of all tensor-invertible objects of $\mathcal T$; ideally then its homogeneous spectrum would retain sufficient information for the resulting comparison map to be injective in more cases. Unfortunately, it is not at all clear -- and probably false in general -- that one may produce in this way a graded commutative ring: as soon as one tries to grade over several invertible objects, some hostile coherence issues appear on the scene to spoil the fun.

\subsection{Results}

In the present paper we solve this difficulty by embracing the enemy, as it were. Instead of trying to construct a \emph{ring} out of the data $\{ \Hom_\mathcal T(\unit, g)\mid g\in \mathcal T \otimes\textrm{-invertible} \}$ which is graded commutative over the Picard \emph{group}, we instead consider this data as defining a \emph{2-ring} which is graded commutative over its Picard \emph{2-group}.
It is actually more natural to consider $\Hom_\mathcal T(g,h)$ for all invertible $g$ and~$h$, and in order to capture the relevant structure we define a \emph{graded commutative 2-ring} (Def.~\ref{defi:grcomm2ring}) to be an essentially small $\Z$-linear category~$\mathcal R$ equipped with an additive symmetric tensor product with respect to which every object is invertible. 

This strategy is successful because, as we will show in Section~\ref{sec:grcomm2rings}, the basic toolkit of affine algebraic geometry generalizes painlessly to this new context. Thus every graded commutative 2-ring~$\mathcal R$ has a Zariski spectrum of homogeneous prime ideals,  $\Spec \mathcal R$, and the assignment $\mathcal R\mapsto \spec \mathcal R$ defines a functor to spectral spaces and spectral maps, in the sense of Hochster~\cite{hochster:prime} (see Theorem \ref{thm:spec}). The theory of localization at multiplicative subsets works perfectly well and provides in particular localizations of~$\mathcal R$ at each prime $\mathfrak p\in \Spec \mathcal R$ and away from any morphism $r\in \Mor \mathcal R$ (Propositions \ref{prop:fractions} and \ref{prop:loc_spec}). 
All of this is due to the observation that, if $r$ and~$s$ are any two composable maps in a graded commutative 2-ring, then $r$ and $s$ ``commute with each other'' up to isomorphisms and twists (see Proposition~\ref{prop:pseudo_comm}).

In Section \ref{sec:comparison} we proceed to apply this generalized commutative algebra to tensor triangular geometry. If $\mathcal T$ is a $\otimes$-triangulated category, we call \emph{a central 2-ring of $\mathcal T$} any graded commutative 2-ring occurring as a full tensor subcategory of~$\mathcal T$. 
We can now state our main abstract result (see Theorem \ref{thm:rho}).

\begin{thm}
\label{thm:main1}
For every central 2-ring $\mathcal R$ of a $\otimes$-triangulated category~$\mathcal T$ there exists a continuous spectral map 
\[
\rho^\mathcal R_\mathcal T\colon \spc (\mathcal T) \to \spec (\mathcal R)
\]
defined by the formula $\rho^\mathcal R_\mathcal T (\mathcal P):=\{ r \in \Mor \mathcal R \mid \cone(r) \not\in \mathcal P \}$ for all $\mathcal P\in \Spc \mathcal T$. 
Moreover $\rho^\mathcal R_\mathcal T$ is natural in an evident sense with respect to pairs $(\mathcal T,\mathcal R)$.
\end{thm}

By viewing ordinary (graded) commutative rings as graded commutative 2-rings, it is easy to see that Theorem \ref{thm:main1} generalizes the construction of Balmer's graded and ungraded comparison maps (\cf\ Examples \ref{ex:comm_2_rings} and \ref{ex:projective}). In view of the freedom of choice for the central 2-ring $\mathcal R$ that Theorem~\ref{thm:main1} offers us, we now have in our hands a whole new phylum of candidate spaces -- the Zariski spectra $\Spec \mathcal R$ -- to help us compute $\Spc \mathcal T$ in examples. More specifically, the geometry of graded commutative 2-rings bridges the gap between the ``mundane'' spectra of commutative rings and the ``exotic'' triangular spectra.

We should mention that, on our way to establishing Theorem \ref{thm:main1}, we prove that every central 2-ring of a \emph{local} tensor triangulated category (\cite{balmer:spec3}*{Def.~4.1}) must be a local graded commutative 2-ring (see Theorem \ref{thm:local_Rtot}). Furthermore, we completely extend Balmer's elegant technique of \emph{central localization} to graded commutative 2-rings (see Theorems \ref{thm:central_loc} and~\ref{thm:fractions_general}).

In Section~\ref{sec:examples} we illustrate our construction with two families of examples, namely (ordinary) graded commutative rings, and schemes with an ample family of line bundles.

Consider a $G$-graded ring~$R$, where $G$ is some abelian group, which is graded commutative with respect to some signing symmetric form $G\times G\to \Z/2$. 
We want to study the tensor triangulated category $\mathcal T:= \D^\perf(R)= \D(R\textrm-\Gr\Mod)^c$ of perfect complexes of graded $R$-modules.
To this end, we note that the companion category $\mathcal C_R$ of~$R$ (see Example \ref{ex:comm_2_rings}~(3)) is a graded commutative 2-ring whose spectrum $\spec \mathcal C_R$ is just $\specgr R$, the homogeneous spectrum of~$R$.
Moreover $\mathcal C_R$ is equivalent, as a graded commutative 2-ring  (essentially via the Yoneda embedding), to the central 2-ring $\mathcal R$ of $\D^\perf(R)$ generated by the twists $R(g)$, $g\in G$, of the ring itself.
With these identifications, we obtain our first application (see Theorem \ref{thm:grcommrings}):

\begin{thm}
\label{thm:main2}
The comparison map of Theorem \ref{thm:main1} yields a homeomorphism
\[
\Spc(\D^\perf(R)) \stackrel{\sim}{\to} \specgr(R)
\]
for every graded commutative ring $R$, identifying Balmer's universal support data in the triangular spectrum with the homological support data in the Zariski spectrum.
\end{thm}

The proof rests on a general abstract criterion for injectivity of $\rho^\mathcal R_\mathcal T$ (see Proposition \ref{prop:injectivity}) and the reduction to the case of a \emph{noetherian} graded commutative ring, which had already been established in~\cite{ivo_greg:graded}*{Theorem 5.1}.
By general tensor triangular geometry  as in \emph{loc.\,cit.\ }we can now immediately translate the previous theorem into the following classification result.

\begin{cor}
For any graded commutative ring~$R$ there is an inclusion-preserving bijection between:
\begin{enumerate}
\item thick subcategories $\mathcal C$ of $\D^\perf(R)$ that are closed under twisting by arbitrary elements $g\in G$, and
\item subsets $S$ of the homogeneous spectrum $\specgr R$ of the form $S=\bigcup_iZ_i$, where each $Z_i$ is closed and has quasi-compact complement in $\specgr R$.
\end{enumerate}
The correspondence maps a twist-closed thick subcategory $\mathcal C$ to the union of the homological supports of its objects, $\bigcup_{X\in \mathcal C} \supph X$, and conversely a given subset~$S$ of the required form is mapped to the subcategory $\{X\in \D^\perf(R)\mid \supph(X)\subseteq S\}$.
\end{cor}

\begin{remark}
We note that, by considering a problem about \emph{ordinary} graded commutative rings~$R$, we were naturally led to consider graded commutative 2-rings: namely the companion category $\mathcal C_R$ and its non-strict incarnation $\mathcal R$ inside of the derived category.
\end{remark}

Next consider a quasi-compact and quasi-separated scheme~$X$. Assume that $X$ admits an ample family of line bundles $\underline{\mathcal L}:=\{\mathcal L_\lambda\}_{\lambda\in \Lambda}$, and denote by $\mathcal R(\underline{\mathcal L})$ the central 2-ring of $\D^\perf(X)$ generated by the ample family. By applying the same abstract injectivity criterion as before, we obtain our second application (see Theorem~\ref{thm:ample}):

\begin{thm}
\label{thm:main3}
The comparison map  associated to the tensor-triangulated category $\D^\perf(X)$ and its central 2-ring $\mathcal R(\underline{\mathcal L})$ yields an injective map 
\[
\rho^{\underline{\mathcal L}}_X 
\colon X \hookrightarrow \Spec(\mathcal R(\underline{\mathcal L}))
\]
which, moreover, is a homeomorphism onto its image.
\end{thm}

By all rights, the morphism $\rho^{\underline{\mathcal L}}_X$ should be geometric, provided we know in which sense to associate some geometry with the spectrum of a graded commutative 2-ring. In this case the geometric point of view is explained in work of Brenner and Schr\"{o}er \cite{brenner-schroer} and our result recovers, via tensor triangular geometry, the map of topological spaces underlying their construction (see Remark~\ref{rem:ffs}).


\subsection{Some recollections and notations}
\label{subsec:inv}
Throughout all categories are understood to be locally small.

Let us recall a few facts and fix some notation about closed symmetric monoidal categories (for further details we refer \eg\ to \cite{kelly-laplaza} or \cite{lewis-may-steinberger}*{III.1}).
In any monoidal category~$\mathcal C$ with tensor $\otimes$ and unit object~$\unit$, we will reserve the letters $\lambda$, $\rho$ and $\alpha$ for the structural coherence natural isomorphisms (left unitor, right unitor, and associator)
\[
\lambda_x \colon \unit \otimes x \stackrel{\sim}{\to} x
\quad,\quad
\rho_x \colon x\otimes \unit \stackrel{\sim}{\to} x
\quad,\quad
\alpha_{x,y,z} \colon x\otimes (y\otimes z) \smash{\stackrel{\sim}{\to}} (x\otimes y)\otimes z 
\]
for all objects $x,y,z\in \mathcal C$.
By \emph{tensor category} we will always mean a symmetric monoidal category, and we will denote the symmetry isomorphism by
\[
\gamma_{x,y}\colon x\otimes y\stackrel{\sim}{\to} y\otimes x \;.
\]
Let $\mathcal C$ be a tensor category. The dual of a dualizable object $x$ of $\mathcal C$ will be denoted by~$x^\vee$. It is determined up to isomorphism by a natural bijection $\mathcal C(x^\vee\otimes -, -)\cong \mathcal C(-, x \otimes -)$, or equivalently, by two maps 
$\eta_{x} \colon \unit \to x\otimes x^\vee$ and $\varepsilon_{x} \colon x^\vee \otimes x\to \unit$ making the following two diagrams commute (where one goes backward along the structure isomorphisms where necessary):
\[
\xymatrix@C=6pt{
 (x\otimes x^\vee) \otimes x  &&
   x \otimes (x^\vee \otimes x)  \ar[ll]_-\alpha^-\sim \ar[d]^{\id_x \otimes \varepsilon_x} \\
\unit \otimes x \ar[r]^-\lambda_-\sim \ar[u]^{\eta_x \otimes \id_x}  &
 x &
  \ar[l]_-\rho^-\sim x \otimes \unit
 }
 \quad\;
 \xymatrix@C=6pt{
x^\vee \otimes ( x \otimes x^\vee) \ar[rr]^-\alpha_-\sim  &&
 (x^\vee \otimes x )\otimes x^\vee \ar[d]^{\varepsilon_x \otimes \id_{x^\vee}}  \\
x^\vee \otimes \unit \ar[r]_-\sim^-\rho \ar[u]^{\id_{x^\vee} \otimes \eta_x} &
 x^\vee &
  \unit \otimes x^\vee \ar[l]^-\sim_-\lambda
 }
\]
These are sometimes called the ``zig-zag identities'' and are nothing but the two triangle identities for the adjunction between $x\otimes -$ and $x^\vee \otimes -$ (slightly disguised). 
The assignment $x\mapsto x^\vee$ extends canonically to morphisms to define a self-duality functor (a contravariant autoequivalence) on the full subcategory of dualizable objects in~$\mathcal C$.

The evaluation  $\varepsilon$ and coevaluation $\eta$ are moreover dinatural, in the sense that for any morphism $f\colon x \to y$, the following two squares are commutative:
\begin{equation} \label{dinat}
\xymatrix{
y^\vee \otimes x \ar[r]^-{\id_{y^\vee} \otimes f} \ar[d]_{f^\vee \otimes \id_x} 
 & y^\vee \otimes y \ar[d]^{\varepsilon_{y}} &  \unit \ar[r]^-{\eta_x} \ar[d]_{\eta_y}
 & x \otimes x^\vee \ar[d]^{f \otimes \id_{x^\vee}} \\
x^\vee \otimes x \ar[r]^-{\varepsilon_{x}} & \unit & y \otimes y^\vee \ar[r]^-{\id_y \otimes f^\vee } & y \otimes x^\vee
}
\end{equation}
(In fact by definition of the functor $(-)^\vee$ the isomorphism $\mathcal C(x^\vee \otimes y,z)\cong \mathcal C(y, x\otimes z)$ is natural in $x,y$ and~$z$, and thus defines an adjunction with parameter. This can be seen to be equivalent to the unit and counit being dinatural as above, see \cite{maclane}*{IX.4}.)

An object $x$ is \emph{invertible} if there exists an object~$x'$ and an isomorphism $x\otimes x'\cong \unit$ (and therefore $x'\otimes x\cong \unit$). If $x$ is invertible it is also dualizable, and indeed, the dual $x^\vee$ provides a canonical choice for an inverse~$x'$, since in this case the unit and counit maps are isomorphisms $\eta\colon \unit \stackrel{\sim}{\to} x\otimes x^\vee$ and $\varepsilon\colon x^\vee\otimes x\stackrel{\sim}{\to} \unit $.

\begin{notation}
In order to alleviate our notational burden, we will often omit  from displayed diagrams all tensor symbols~$\otimes$ and subscripts for natural transformations, and we will often denote an identity map $\id_x$ by the object~$x$. 
Thus for instance, if there is no danger of confusion, we will simply write 
\begin{equation*} 
\xymatrix{
y^\vee  x \ar[r]^-{y^\vee  f} \ar[d]_{f^\vee  x} & y^\vee  y \ar[d]^{\varepsilon} & \unit \ar[r]^-{\eta} \ar[d]_{\eta}
 & x  x^\vee \ar[d]^{f x^\vee  }\\
x^\vee  x \ar[r]^-{\varepsilon} & \unit & y y^\vee  \ar[r]^-{y f^\vee } & y x^\vee  
}
\end{equation*}
for the dinaturality squares~\eqref{dinat}.
Occasionally we will also omit the associativity and unit isomorphisms, as justified by Mac Lane's coherence theorem.
\end{notation}

\section{Graded commutative 2-rings}
\label{sec:grcomm2rings}

The following definition is commonly understood to be a sensible categorification of the concept of abelian group, see \eg\ \cite{baez-lauda:higher}, \cite{dupont:thesis} and the many references therein.

\begin{defi}
A \emph{symmetric 2-group} is a symmetric monoidal groupoid in which every object is invertible for the tensor product. 
\end{defi}

\begin{examples}
\label{ex:picard_category}
\begin{enumerate}

\item
Every abelian group $G$ can be considered as a discrete symmetric 2-group, \ie, as the discrete category with object set $\obj G=G$ equipped with the strict symmetric tensor product $g\otimes h = g+h$ and~$\unit=0$.

\item
The \emph{Picard 2-group} (or \emph{Picard groupoid}, \emph{Picard category}) of a symmetric monoidal category~$\mathcal C$ is the symmetric 2-group obtained from~$\mathcal C$ by considering the (non-full) monoidal subcategory of all invertible objects and isomorphisms between them. 

\end{enumerate}
\end{examples}

Let $\mathcal G$ be an essentially small symmetric 2-group.

\begin{defi}
\label{defi:grcomm2ring}
A \emph{$\mathcal G$-graded commutative 2-ring~$\mathcal R$}  is a symmetric monoidal $\Z$-category~$\mathcal R$ equipped with a symmetric monoidal functor $\mathcal G\to \mathcal R$ which is surjective on objects (thus $\mathcal R$ is essentially small and its objects are invertible). 
\end{defi}

Typically $\mathcal G\to \mathcal R$ will simply be the inclusion of the Picard 2-group of $\mathcal R$, in which case we will simply speak of a \emph{graded commutative 2-ring}.
Hence we will not distinguish notationally between the objects of $\mathcal R$ and those of~$\mathcal G$; they will be written $g,h,\ell,\ldots\in \mathcal G$ and thought of as ``degrees''.
We will also think of the morphisms $r$ of $\mathcal R$ as ``elements'', so for instance we will tend to write $r\in \mathcal R$ rather than $r\in \Mor \mathcal R$.

We make the convention that natural structure maps, for example the left and right unitors, are written in the direction in which they occur; when defining a composite using such maps it is understood that the composite is obtained by going backward (taking the inverse) along any such arrows which are in the ``wrong'' direction. Given a morphism $r\colon g\to h$ in $\mathcal{R}$ and an object $\ell \in \mathcal{R}$, we will refer to the morphisms
 $r\otimes \ell =r\otimes \id_{\ell} \colon g\otimes \ell \to h \otimes \ell$ and $\ell\otimes r\colon \ell\otimes g\to \ell\otimes h$ as the \emph{right twist}, respectively \emph{left twist}, \emph{of~$r$ by~$\ell$}.

\begin{examples}
\label{ex:comm_2_rings}
\begin{enumerate}
\item 
The zero category $\{0\}$ is a $\mathcal G$-graded commutative 2-ring for any choice of~$\mathcal G$, in a unique way. 
There is an evident many-objects version of $\{0\}$ for any choice of object class with a distinguished element (which will be elected to be the tensor unit), which is of course monoidally equivalent to $\{0\}$ and is graded over a suitable trivial symmetric 2-group.   

\item Let $R$ be a commutative ring. We may consider $R$ as a $\Z$-linear category with a single object~$*$ whose endomorphism ring is~$R$, and we may equip this category with the strict symmetric tensor product $r\otimes s:= rs$, for $r,s\in R$. Thus we can view $R$ as a commutative 2-ring graded by the trivial (symmetric 2-)group.

\item   (See \cite{ivo_greg:graded} for details.)
 Let $G$ be an abelian group, and let $R$ be a $G$-graded $\epsilon$-commutative ring, for some signing symmetric form $\epsilon\colon G\times G\to \Z/2$. 
The \emph{companion category} of $R$, denoted~$\mathcal C_R$, is the small $\Z$-linear category with object set~$\obj \mathcal C_R:=G$, with Hom groups $\mathcal C_R(g,h):=R_{h-g}$, and with composition given by the multiplication of~$R$. Multiplication also yields a strict symmetric tensor product, which on objects $g,h\in G$ is the sum $g\otimes h:=g+h$ and on two maps $r\colon g\to h$ and $r'\colon g'\to h'$ is given by the formula $r\otimes r':=(-1)^{\epsilon(g,g'-h')}rr'$. 
Thus the companion category $\mathcal C_R$ is a $G$-graded commutative 2-ring, where $G$ is seen as a discrete symmetric 2-group. 
Its tensor structure extends to the category $\Ab^{\mathcal C_R}$ of additive functors, and this is tensor equivalent to $R$-$\GrMod$, the tensor category of left graded modules. 
Thus $R$ and $\mathcal C_R$ are Morita tensor equivalent, and once again we can think of $R$ as being a graded symmetric 2-ring, namely~$\mathcal C_R$. 

\item Let $X$ be a scheme and suppose that $\{\mathcal{L}_i\; \vert \; i\in I\}$ is a family of line bundles on $X$. Then the full subcategory of quasi-coherent $\mathcal{O}_X$-modules with objects all finite tensor products of the $\mathcal{L}_i$ and their inverses is a symmetric 2-ring. It is graded over the corresponding subcategory of the Picard 2-group of~$X$.
\end{enumerate}
\end{examples}

\begin{defi}
Let $\iota\colon \mathcal G\to \mathcal R$ and $\iota' \colon \mathcal G'\to \mathcal R'$ be two graded commutative 2-rings. 
A \emph{morphism} $F: \mathcal R\to \mathcal R'$ is a square of $\otimes$-functors
\[
\xymatrix{
\mathcal G \ar[r]^-{\iota} \ar[d]_F &
 \mathcal R \ar[d]^F \\
\mathcal G' \ar[r]^-{\iota'} &
 \mathcal R'
}
\]
such that $F\colon \mathcal R\to \mathcal R'$  is additive and such that there exists an isomorphism $F\iota \cong \iota'F$.
If $\iota$ and $\iota'$ are inclusions (\eg\ when $\mathcal R$ and $\mathcal R'$ are graded by their Picard 2-groups), then we only talk about the tensor functor $F\colon \mathcal R\to \mathcal R'$ and assume that $\mathcal G\to \mathcal G'$ is its restriction.
\end{defi}


Among the numerous examples of morphisms that will be used later on, we mention those produced by localization (see \S\ref{subsec:loc}), and the inclusions between different central 2-rings in a tensor-triangulated category (see \S\ref{subsec:central2rings}).  

\subsection{Pseudo-commutativity}
Let $\mathcal R$ be a graded commutative 2-ring.

\begin{defi}
Let $r\in \mathcal R$. By a \emph{translate} of~$r$ we will mean any morphism of~$\mathcal R$ obtained from $r$ by the operations of taking twists of morphisms and composing with isomorphisms on either side, in any combination.
\end{defi}

\begin{lemma}
\label{lemma:translate}
Every translate of~$r$ is isomorphic to a left twist and also to a right twist of~$r$, that is, has the form $u(g\otimes r) v$ and also the form $u' (r\otimes g') v'$ for some objects $g,g'\in \mathcal G$ and some isomorphisms $u,v,u',v'$. Moreover, translation (the relation ``$\,\tilde r$ is a translate of~$r$'') is an equivalence relations on the morphisms of~$\mathcal R$.
\end{lemma}
\begin{proof}
Every right twist~$s\otimes h$ of a morphism~$s$ by an object~$h$ is isomorphic to the corresponding left twist $h\otimes s$, by the symmetry of the tensor product.
Since twisting preserves isomorphisms and an iteration of twists is again (up to isomorphism) a twist, we conclude that every translate of~$r$ can be brought to either of the two forms.
It is similarly easy to check that translation is an equivalence relation.
\end{proof}

\begin{example}
\label{ex:dual_translate}
For every map $r\colon g\to h$ in $ \mathcal R$, the dinaturality square 
\[
\xymatrix{ h^\vee g \ar[d]_{r^\vee g} \ar[r]^-{h^\vee r} & h^\vee h \ar[d]^{\varepsilon}_\sim \\
g^\vee g \ar[r]^-{\varepsilon}_-\sim & \unit
}
\]
shows that $r$ and $r^\vee$ are translates of each other.
\end{example}

The next all-important proposition states that, in a graded commutative 2-ring, composition is commutative \emph{up to translation}. Later on we will also use a slight generalization of the same argument, when the need will arise to consider algebras over graded commutative 2-rings  (see \S\ref{subsec:loc_alg}). 

\begin{prop}
\label{prop:pseudo_comm}
Let $r\colon g \to h$ and $s\colon h\to \ell$ be two composable morphisms of~$\mathcal R$. 
Then there exists  in $\mathcal R$ a commutative diagram of the form 
\[
\xymatrix{ 
g \ar@/_/[d]_u^\sim \ar[r]^-r &
 h \ar[r]^-s  &
   {\ell}  \\
hh^\vee g \ar[r]^-{s h^\vee g} &
  \ell h^\vee g \ar[r]^-{\ell h^\vee r} & \ell h^\vee h   \ar@/_/[u]_v^\sim
}
\]
and where $u$ and~$v$ are some isomorphisms. In particular, $sr= r' s'$ for some translates $r'$ of~$r$ and $s'$ of~$s$.
\end{prop}

\begin{proof}
We may construct the following commutative diagram:
\begin{equation} \label{eq:pseudo_comm} 
\xymatrix{
 *+[F]{g}  \ar[r]^-{r} &
    *+[F]{h}  \ar[r]^-{s} &
     *+[F]{\ell} & \\
 g\unit \ar[u]^\sim_{\rho} \ar[r]^-{r\unit} &
  h \unit  \ar[u]^{\rho} \ar[r]^-{s\unit} &
    \ell \unit  \ar[u]_-\sim^-{\rho} & \\
 gg^\vee g \ar[u]_{g \varepsilon_{g}}^\sim  \ar[r]^-{rg^\vee g} &
  h g^\vee g \ar[u]^{h \varepsilon_{g}} \ar[r]^-{sg^\vee g}  &
   { \ell g^\vee g } \ar[u]_\sim^{\ell \varepsilon_{g}} \ar[r]^-\sim_-{\ell \varepsilon_{g}}  & {\ell \unit} \\
\unit g \ar[u]^\sim_{\eta_{g}\, g} \ar[r]_-\sim^-{\eta_{h}\, g} & 
 *+[F]{ hh^\vee g } \ar[u]_{hr^\vee g} \ar[r]_-{sh^\vee g}   &
  *+[F]{\ell h^\vee g} \ar[u]_{\ell r^\vee g} \ar[r]_-{\ell h^\vee r} & *+[F]{\ell h^\vee h} \ar[u]_\sim^{\ell \varepsilon_{h}}
}
\end{equation}
Indeed, the top two squares commute by the naturality of~$\rho$;
the bottom-left and bottom-right ones by applying $-\otimes g$, resp.\ $\ell\otimes -$,  to appropriate dinaturality squares;
the remaining three by the (bi)\-functor\-ial\-ity of the tensor product.  
Note that all the morphisms labeled $\sim$ are isomorphisms, because the objects $g$ and $h$ are $\otimes$-invertible.
Now it suffices to compose the maps on its outer frame between the boxed objects, 
in order to obtain a diagram as claimed.
\end{proof}

\begin{cor}
\label{cor:composite_iso}
In a graded commutative 2-ring, any (right or left) composite of a non-invertible morphism is again  non-invertible.
\end{cor}

\begin{proof}
We prove the converse: if $r\colon g\to h$ and $s\colon h \to \ell$ are two composable maps such that their composite $sr$ is invertible, then $s$ and $r$ are invertible. 
Of course, it suffices to show that $s$ is invertible.
As in any category, for this it suffices to show that $s$ has both a right inverse and a left inverse.
Since $sr$ is invertible, $r(sr)^{-1}$ is clearly a right inverse of~$s$.
To find a left inverse, consider a commutative diagram as in Proposition~\ref{prop:pseudo_comm}:
\[
\xymatrix{ 
g \ar@/_/[d]_u^\sim \ar[r]^-r &
 h \ar[r]^-s  &
   {\ell}  \\
\bullet \ar[r]^-{s'} &
  \bullet \ar[r]^-{r'} & \bullet   \ar@/_/[u]_v^\sim
}
\]
Here $s'$ is a twist of~$s$, $r'$ is a twist of~$r$, and $u$ and $v$ are invertible.
Since $sr$ is invertible, so is $r's'$. 
Hence $(r's')^{-1}r'$ is a left inverse of~$s'$. 
Since twisting is an equivalence of categories, this shows that $s$ also has a left inverse.
\end{proof}

\subsection{Homogeneous ideals}

\begin{defi}
A \emph{homogeneous ideal} $\mathcal I$ of~$\mathcal R$ is a ($\Z$-linear categorical) ideal of morphisms of~$\mathcal R$ which is closed under tensoring with arbitrary morphisms of~$\mathcal R$. 
In other words, it is a collection of subgroups $\mathcal I(g,h)\subseteq \mathcal R(g,h)$ for all $g,h\in \mathcal G$, satisfying the two properties
\begin{enumerate}
\item $\mathcal R(h,h') \circ \mathcal I(g,h)\circ \mathcal R(g',g) \subseteq \mathcal I(g',h')$
\item $\mathcal R(g',h')\otimes \mathcal I(g,h) \subseteq \mathcal I(g'\otimes g, h'\otimes h)$
and $\mathcal I(g,h) \otimes \mathcal R(g',h') \subseteq \mathcal I(g\otimes g', h\otimes h')$
\end{enumerate}
for all $g',h'\in \mathcal G$. 
\end{defi}

We note that as in a symmetric 2-ring there are no direct sums all maps are ``homogeneous of some degree'', so we do not need to impose a homogeneity condition on generators, as is the case for usual graded rings (cf.\ Example \ref{ex:comm_2_rings}~(3)).
Nonetheless, here as elsewhere we have borrowed the terminology -- along with the intuition -- from graded rings.

\begin{remarks}
\label{rem:ideals}
We collect here some immediate observations on homogeneous ideals.
\begin{enumerate}

\item 
It follows from the essential smallness of $\mathcal R$ that there is only a set of homogeneous ideals of~$\mathcal R$.

\item If $\mathcal I$ is a homogeneous ideal of~$\mathcal R$, then $\otimes$ descends to a tensor structure on the additive quotient category $\mathcal R/\mathcal I$. The quotient functor $\mathcal R\to \mathcal R/\mathcal I$ is symmetric monoidal and thus preserves invertible objects, so $\mathcal R/\mathcal I$ is again a graded commutative 2-ring.

Moreover, the homogenous ideals of $\mathcal R/\mathcal I$ are in bijection with the homogeneous ideals of $\mathcal R$ containing~$\mathcal I$, in the usual way.
Conversely, the kernel (on maps) of any morphism $\mathcal R\to \mathcal R'$ of graded commutative 2-rings is a homogeneous ideal.

\item A homogeneous ideal is proper iff it contains no isomorphism, iff it does not contain any identity map, iff it does not contain the identity map of~$\unit$.

\item If the homogeneous ideal $\mathcal I$ contains a morphism $r$ then it also contains all the translates of~$r$.

\item 
Since duals are translates (Example \ref{ex:dual_translate}), we see in particular that every homogeneous ideal is self-dual: $\mathcal I=\mathcal I^\vee$.

\item
It follows from the factorizations 
\[
\xymatrix{
 g g' \ar[d]_{r {g'}} \ar[dr]|{r r'} \ar[r]^-{g  r'} &
  g  h' \ar[d]^{r  {h'}} \\
 h g' \ar[r]_-{h  r' \; } & h h'
}
\]
of $r\otimes r'$, for any $r\in \mathcal R(g,h)$ and $r'\in \mathcal R(g',h')$,
that an ideal $\mathcal I$ of $\mathcal R$ is closed under tensoring with arbitrary maps iff it is closed under tensoring with arbitrary objects (\ie, with identity maps). By symmetry, it suffices that this holds on one side, \ie\ that  $g\otimes r \in \mathcal I$ for all $g\in \mathcal G$ and $r\in \mathcal I$. 

\item Since all objects of $\mathcal R$ are invertible, we see from the commutative square
\[
\xymatrix{
g^\vee  g  h \ar[r]^-{g^\vee  g  r } \ar[d]_{\varepsilon_{g}\,h}^\sim &
 g^\vee  g \ell \ar[d]_\sim^{\varepsilon_{g}\, \ell} \\
h \ar[r]^-r &
 \ell
}
\]
and (6) (and the associativity of~$\otimes$) that the second condition in the definition of a homogeneous ideal is equivalent to
\[
r\in \mathcal I \; \Leftrightarrow \; g\otimes r \in \mathcal I 
\quad  \textrm{ for all maps } r \textrm{ and objects } g 
\]
and also to
\[
r\in \mathcal I \; \Rightarrow \; g\otimes r \in \mathcal I 
\quad  \textrm{ for all maps } r \textrm{ and objects } g \,. 
\]
\end{enumerate}
These remarks will be used repeatedly without further mention.
\end{remarks}

The set of all homogeneous ideals of $\mathcal R$ forms a poset with respect to inclusion. This poset is in fact a complete lattice, where meets $\bigwedge_\lambda \mathcal I_\lambda = \bigcap_\lambda  \mathcal I_\lambda$ are just intersections, and therefore joins can be given by 
$\bigvee_\lambda \mathcal I_\lambda = \bigcap_\lambda \{ \mathcal J \mid \mathcal I_\lambda \subseteq \mathcal J  \}$.
A much more useful formula is provided by the following lemma.

\begin{lemma}
\label{lemma:unions}
The join of any family $\{\mathcal I_\lambda\}_{\lambda\in \Lambda}$ of homogeneous ideals of~$\mathcal R$ is given by their Hom-wise sum
$\left(\bigvee_\lambda \mathcal I_\lambda\right)(g,h) =\sum_\lambda \mathcal I_\lambda (g,h) $ for all $g,h\in \mathcal G$.
\end{lemma}
\begin{proof}
It suffices to prove that the Hom-wise sum of the $\mathcal I_\lambda$'s is a homogeneous ideal, since any other homogeneous ideal containing the $\mathcal I_\lambda$'s will have to contain it as well. 
Let $g,h\in \mathcal G$. By definition, $\sum_\lambda \mathcal I_\lambda (g,h)$ is a subgroup of $\mathcal R(g,h)$. Moreover, if $r_{\lambda_1}+\ldots +r_{\lambda_n}\in \sum_\lambda \mathcal I_\lambda (g,h)$ is an arbitrary element then its twists $g\otimes (r_{\lambda_1}+\ldots +r_{\lambda_n})= g\otimes r_{\lambda_1}+\ldots +g\otimes r_{\lambda_n}$ and its multiples $s (r_{\lambda_1}+\ldots +r_{\lambda_n})=sr_{\lambda_1}+\ldots +sr_{\lambda_n}$ and $(r_{\lambda_1}+\ldots +r_{\lambda_n})t= r_{\lambda_1}t+\ldots +r_{\lambda_n}t$ are again in $\sum_\lambda \mathcal I_\lambda (g,h)$, since each $\mathcal I_\lambda$ is closed under these operations. Thus $\bigcup_{g,h}\sum_\lambda \mathcal I_\lambda (g,h)$ is an ideal as claimed completing the proof.
\end{proof}

Therefore we will also write $\cap$ for meets and $+$ or $\sum$ for joins.

\begin{prop}
\label{prop:principal}
Let $r\colon g\to h$ be any morphism of $\mathcal R$. The principal homogeneous ideal $\langle r\rangle\subseteq \mathcal R$ generated by~$r$ admits the following explicit description:
for all $g',h'\in \mathcal G$, its component $\langle r\rangle (g',h')$ consists precisely of the morphisms of the form $s(\ell \otimes r)u$ for some object $\ell \in \mathcal G$, some isomorphism $u\colon g'\stackrel{\sim}{\to} \ell \otimes g$ and some morphism $s\colon \ell\otimes h\to h'$. Dually, it also consists precisely of the morphisms of the form $v(r \otimes k)t$ for some object $k\in \mathcal G$, map $t\colon g'\to g \otimes k $ and isomorphism $v\colon h\otimes k\stackrel{\sim}{\to} h'$.
\end{prop}

\begin{proof}
We only prove the first claim, the second one being dual. 
Let $\mathcal I(g',h')$ be the set of maps $g'\to h'$ of the form $s(\ell \otimes r)u$, as above. 
It suffices to show that $\mathcal I:= \bigcup_{g',h'}\mathcal I(g',h')$ is a homogeneous ideal containing~$r$, since clearly $\mathcal I\subseteq \langle r\rangle$.
Note that $r$ has the required form, for instance because of the commutative square 
\[
\xymatrix{
\unit g \ar[d]^\lambda_\sim  \ar[r]^-{\unit r } & \unit h \ar[d]_\sim^\lambda \\
 g\ar[r]^-r & h\,.
  }
\]
Moreover $\mathcal I$ is evidently closed under twists and under compositions on the left, hence it remains only to show that it is closed under sums and compositions on the right.

Thus consider $r':= s(\ell\otimes r)u\in \mathcal I (g', h')$, with $u$ invertible, and let $t\in \mathcal R(g'', g')$ be some map. 
By Proposition \ref{prop:pseudo_comm}, their composition $r'  t$ is equal to $t'  r''$ for some translate $t'$ of $t$ and some translate $r''$ of~$r'$. 
But the latter is immediately seen to be a morphism of the form $\tilde t (\tilde{\ell}\otimes r) \tilde u$ with~$\tilde u$ invertible, so it belongs to $\mathcal I(g'', h')$. This shows that $\mathcal I$ is closed under composition on the right.

Finally, let 
\[
r'= \bigg(
 \xymatrix{ g' \ar[r]^-u_-\sim & \ell\otimes g \ar[r]^-{\ell\otimes r} & \ell \otimes h \ar[r]^-{s} & h'  } 
 \bigg)
\]
and 
\[
r''=\bigg( 
\xymatrix{ g' \ar[r]^-{\tilde u}_-\sim & \tilde \ell\otimes g \ar[r]^-{\tilde \ell \otimes r} & \tilde \ell \otimes h \ar[r]^-{\tilde s} & h'  } \bigg)
\]
be two morphisms in $\mathcal I(g',h')$. Since $u$ and $\tilde u$ are isomorphisms and $g$ is tensor-invertible, we deduce the existence of an isomorphism $\varphi\colon \tilde \ell \cong g'\otimes g^\vee \cong \ell$ between $\tilde \ell$ and~$\ell$, which we can use to constuct a commutative diagram
\[
\xymatrix@R=7pt{
& \tilde \ell\otimes g \ar[dd]^{\varphi\otimes g} \ar[r]^-{\tilde \ell \otimes r}  &
 \tilde \ell \otimes h \ar[dr]^-{\tilde s} \ar[dd]^{\varphi \otimes h} & \\
g' \ar[ru]^-{\tilde u} \ar[rd]_-{v\, :=} &
  &
   & h' \\
 & \ell \otimes g \ar[r]^-{\ell \otimes r} & \ell \otimes h \ar[ur]_-{=:\,t}
  & 
}
\]
where $v$ is again an isomorphism. Thus we may write $r''= t(\ell \otimes r)v$. 
Since $-\otimes g$ is an endo-equivalence of~$\mathcal R$, the automorphism $uv^{-1}\colon \ell \otimes g\stackrel{\sim}{\to} \ell \otimes g$ must have the form $\psi \otimes g$ for some automorphism $\psi$ of~$\ell$.
We thus obtain a commutative diagram
\[
\xymatrix@R=7pt{
&  \ell\otimes g \ar[dd]^{\psi\otimes g}_\sim \ar[r]^-{ \ell \otimes r}  &
  \ell \otimes h \ar[r]^-{t} \ar[dd]^{\psi \otimes h}_\sim &
  h' \\
g' \ar[ru]^-{v} \ar[rd]_-{u} &
  &
   & \\
 & \ell \otimes g \ar[r]^-{\ell \otimes r} & \ell \otimes h \ar[r]^-{ s}
  & h'
}
\]
where the upper composition $g'\to h'$ is $r''$ and the lower one is~$r'$.
Now we may use it to compute
\begin{align*}
r'+r''
& = s(\ell \otimes r)u + t(\ell \otimes r)v \\
& = s(\ell \otimes r)u + t (\psi \otimes h)^{-1} (\ell \otimes r)u \\
& = (s+ t(\psi \otimes h)^{-1}) (\ell \otimes r) u \\
& \in \mathcal I(g',h') \,,
\end{align*}
which shows that $\mathcal I(g',h')$ is closed under sums.
Since clearly it also contains the zero map, this concludes the proof that $\mathcal I$ is a homogeneous ideal.
\end{proof}

We now introduce one last family of homogeneous ideals.

\begin{defi}\label{ex:spectrum}
The \emph{annihilator} of a morphism~$s$ of~$\mathcal R$, denoted by $\Ann_\mathcal R(s)$, is the homogeneous ideal of $\mathcal R$ generated by all the $r\in \mathcal R$ such that $r\circ s=0$.
\end{defi}

This definition looks as if it should be called the \emph{left} annihilator of~$s$, but the next lemma shows that, as for commutative rings, there is actually no difference between left and right annihilators.

\begin{prop}
\label{prop:ann}
The annihilator of $s\in \mathcal R$ has the explicit descriptions 
\begin{align*}
\Ann_{\mathcal R}(s) &=\{r\in \Mor \mathcal R \mid \exists \textrm{ a translate }\tilde r\textrm{ of } r \textrm{ s.t.\ } \tilde rs=0  \} \\
&= \{r\in \Mor \mathcal R \mid \exists g\in \mathcal G \textrm{ and } \exists \textrm{ an isomorphism } u \textrm{ s.t.\ }  (g\otimes r)us=0  \}
\end{align*}
and
\begin{align*}
\Ann_{\mathcal R}(s) &=\{r\in \Mor \mathcal R \mid \exists \textrm{ a translate }\tilde r\textrm{ of } r \textrm{ s.t.\ } s\tilde r=0  \} \\
&= \{r\in \Mor \mathcal R \mid \exists g\in \mathcal G \textrm{ and } \exists \textrm{ an isomorphism } u \textrm{ s.t.\ }  s u (g\otimes r)=0  \} \,.
\end{align*}
In particular, by symmetry, $\Ann_\mathcal R(s)$ is also equal to the homogenous ideal of $\mathcal R$ generated by the maps $r\in \mathcal R$ such that $sr=0$.
\end{prop}

\begin{proof}
 Observe that the second equality in both parts of the statement is immediate from Lemma \ref{lemma:translate}. Thus it is sufficient to prove the first equality in each statement; as they are similar we only give a proof of the top one.

Let $\mathcal I:=\{r\in \Mor \mathcal R \mid \exists \textrm{ a translate }\tilde r\textrm{ of } r \textrm{ s.t.\ } \tilde rs=0  \}$.
Since $\Ann_\mathcal R(s)$ is a homogeneous ideal it is closed under translates and thus by definition it must contain~$\mathcal I$. 
To prove the reverse inclusion, it suffices to show that $\mathcal I$ is a homogeneous ideal. If $r$ is in $\mathcal{I}$ then $(g\otimes r)us=0$ for some object~$g$ and isomorphism~$u$. Given $h\in \mathcal{G}$ we see that $h\otimes r\in \mathcal{I}$ i.e., $\mathcal{I}$ is closed under left twists, by considering
\[
((g\otimes h^\vee) \otimes(h\otimes r))us
\cong (g\otimes r)us =0.
\]
Similarly we see that for a morphism $t$ which can be postcomposed with $r$ there is an equality
\[
(g\otimes (tr))us 
= (g\otimes t)(g \otimes r)us
=0
\,,
\]
showing that $\mathcal I$ is closed under left compositions.

To show that $\mathcal I$ is closed under sums we use the same reasoning as in the previous proposition.
Let $r_1,r_2\in \Ann_{\mathcal R}(s)(g,h)$. By definition of $\mathcal I(g,h)$ and by Lemma \ref{lemma:translate}, this means that there exist objects $\ell_1,\ell_2$ and isomorphisms $v_1,v_2$ such that $(\ell_1\otimes r_1)v_1s=0$ and $(\ell_2\otimes r_2)v_2s=0$. 
In particular $v_1$ and $v_2$ have the same domain, and since $g$ is tensor-invertible we deduce the existence of an isomorphism $\varphi\colon \ell_2\stackrel{\sim}{\to}\ell_1$. Hence we may define an isomorphism $w$ fitting into the following commutative diagram.
\[
\xymatrix@R=7pt{
&&  \ell_2\otimes g \ar[dd]_\sim^{\varphi\otimes g} \ar[r]^-{ \ell_2 \otimes r_2}  &
  \ell_2 \otimes h \ar[dd]_\sim^{\varphi \otimes h}  \\
\bullet \ar[r]^-s & \bullet \ar[ru]^-{v_2} \ar[rd]_-{w\, :=} &
  & \\
 && \ell_1 \otimes g \ar[r]^-{\ell_1 \otimes r_2} & \ell_1 \otimes h 
}
\]
Since  $(\ell_2\otimes r_2)v_2s=0$ by hypothesis, we deduce moreover that $(\ell_1 \otimes r_2)ws=0$.
And since $-\otimes g$ is an endo-equivalence of $\mathcal R$, the automorphism $v_1w^{-1}$ of $\ell_1\otimes g$ must have the form $\psi\otimes g$ for some automorphism $\psi$ of~$\ell_1$. Thus we obtain the following commutative diagram:
\[
\xymatrix@R=7pt{
&& \ell_1\otimes g \ar[dd]_\sim^{\psi \otimes g} \ar[r]^-{ \ell_1 \otimes r_2}  &
  \ell_1 \otimes h \ar[dd]_\sim^{\psi \otimes h}  \\
\bullet \ar[r]^-s & \bullet \ar[ru]^-{w} \ar[rd]_-{v_1} &
  & \\
 && \ell_1 \otimes g \ar[r]^-{\ell_1 \otimes r_2} & \ell_1 \otimes h 
}
\]
This allows us to compute
\begin{align*}
(\ell_1\otimes (r_1+r_2))v_1s
& = \underbrace{(\ell_1\otimes r_1)v_1s}_{0} + (\ell_1\otimes r_2)v_1s  \\
& = (\ell_1 \otimes r_2)(\psi\otimes g)ws \\
& = (\psi\otimes h) \underbrace{(\ell_1 \otimes r_2) ws}_{0} =0 \,,
\end{align*}
which shows that $r_1+r_2$ belongs to $\mathcal I$, as wished.

Finally, it remains to verify that $\mathcal I$ is also closed under composition on the right, and this follows easily from Proposition~\ref{prop:pseudo_comm}. 
More precisely, the following claim is an immediate consequence of Proposition~\ref{prop:pseudo_comm}:
\begin{description}
\item[Claim] For any two maps $a,b \in \mathcal R$, we have the following equivalence:\\
$a' b=0$ for some translate $ a'$ of~$a$
$\quad \Leftrightarrow\quad$
$b  a''=0$ for some translate $a''$ of~$a$.
\end{description}

Therefore, if we assume that $r's=0$ for some translate $ r'$ of~$r$, then $sr''=0$ for some (other) translate~$r''$ of~$r$, say $r''= w(\ell\otimes r)w'$ for an object~$\ell$ and isomorphisms $w,w'$. 
But this implies $sw(\ell \otimes r)=0$ and therefore also $sw(\ell \otimes rt)= sw(\ell \otimes r)(\ell \otimes t) =0$. In other words, we have $s a'$ for some translate $a'$ of $a:=rt$. By applying the claim once again, we see that $a''s=0$ for some other translate $a''$ of $rt$. This shows that $\mathcal I$ is closed under composition on the right, as required. 
Hence $\Ann_{\mathcal R}(s)=\mathcal I$.
\end{proof}

\subsection{Products of ideals}
In this subsection we explain how the lattice of homogeneous ideals in $\mathcal R$ --- let us denote it by $\Id(\mathcal R)$ --- is a \emph{commutative ideal lattice}, in the sense of Buan, Krause and Solberg~\cite{bks:ideal}.
This observation provides a quick and conceptual way of defining the Zariski spectrum of~$\mathcal R$.

\begin{lemma}
\label{lemma:products}
Let $\mathcal I,\mathcal J$ be two homogeneous ideals. Then their categorical product
\[
\mathcal I\circ \mathcal J = \{ s_1t_1+ \ldots + s_nt_n \mid s_1,\ldots,s_n\in \mathcal I , t_1,\ldots,t_n\in \mathcal J, n\geq0\}
\] 
and their tensor product
\[
\mathcal I\otimes \mathcal J = \langle \{ 
s\otimes t \mid s\in \mathcal I, t\in \mathcal J \}\rangle
\]
define the same homogeneous ideal, that we will simply denote by $\mathcal I\mathcal J$. It follows in particular that $\mathcal I\mathcal J=\mathcal J\mathcal I$.
\end{lemma}

\begin{proof}
Clearly $\mathcal I\circ \mathcal J$ is a homogeneous ideal, and it follows from $s\otimes t= (s\otimes \id)(\id \otimes t)$ that it contains $\mathcal I\otimes \mathcal J$. On the other hand, 
consider the following truncated version of the commutative diagram \eqref{eq:pseudo_comm}:
\begin{equation*} 
\xymatrix{
 *+[F]{g}  \ar[r]^-{t} \ar@/^4ex/[rr]^-{st} &
    {h}  \ar[r]^-{s} &
     *+[F]{\ell}  \\
 g\unit \ar[u]^\sim_{\rho} \ar[r]^-{t\unit} &
  h \unit  \ar[u]^{\rho} \ar[r]^-{s\unit} &
    \ell \unit  \ar[u]_-\sim^-{\rho}  \\
 gg^\vee g \ar[u]_{g \varepsilon}^\sim  \ar[r]^-{tg^\vee g} &
  h g^\vee g \ar[u]^{h \varepsilon} \ar[r]^-{sg^\vee g}  &
   *+[F]{ \ell g^\vee g } \ar[u]_\sim^{\ell \varepsilon}    \\
\unit g \ar[u]^\sim_{\eta g} \ar[r]_-\sim^-{\eta g} & 
 *+[F]{ hh^\vee g }  \ar[u]_{h t^\vee g} \ar@/_3ex/[ur]_{s \,\otimes\, t^\vee g}   &   
}
\end{equation*}
We read off its outer frame that every composition $st$ ($s\in \mathcal I, t\in \mathcal J$) has the form $u(s \otimes t^\vee \otimes g)v$ for some isomorphisms $u,v$ and is  therefore contained in the homogeneous ideal generated by the tensor products $s' \otimes t'$ ($s'\in \mathcal I, t'\in \mathcal J$), since $t^\vee \otimes g\in \mathcal J$.
Hence $\mathcal I\circ \mathcal J=\mathcal I\otimes \mathcal J$.
The symmetry $\mathcal I\otimes \mathcal J=\mathcal J\otimes \mathcal I$ is obvious from $s\otimes t=\gamma( t\otimes s)\gamma$.
\end{proof}

\begin{lemma}
\label{lemma:cpt_ideals}
The compact elements $\mathcal I\in \Id(\mathcal R)$ (\ie, those for which $\mathcal I\subseteq  \bigvee_\alpha \mathcal I_\alpha$ always implies $\mathcal I\subseteq  \bigvee_{\alpha'} \mathcal I_{\alpha'}$ for some finite subset of indices) are precisely the finitely generated ideals: $\mathcal I=\langle r_1,\ldots, r_n\rangle$.
\end{lemma}

\begin{proof}
If $\mathcal I$ is finitely generated, then it is compact by the sum description of joins.
Conversely, as $\mathcal I= \sum_{r\in \mathcal I}\langle r\rangle$ then if $\mathcal I$ is compact it must be finitely generated.
\end{proof}

\begin{prop}
The poset $\Id(\mathcal R)$ of homogeneous ideals of~$\mathcal R$, ordered by inclusion and equipped with the pairing $(\mathcal I,\mathcal J)\mapsto \mathcal I\mathcal J$, 
is an ideal lattice.
\end{prop}

\begin{proof}
We need to verify the axioms (L1)-(L5) of \cite{bks:ideal}*{Definition~1.1}, and this is quite straightforward.
We have already seen that $\Id(\mathcal R)$ is complete, and it is compactly generated since $\mathcal I= \sum_{r\in \mathcal I}\langle r\rangle$ holds for every~$\mathcal I$. By Lemma \ref{lemma:cpt_ideals},  $1=\langle \id_\unit\rangle$ is compact and the product of two compact elements is compact: writing $\mathcal I=\langle I\rangle$ and $\mathcal J=\langle \mathcal J\rangle$ for finite sets $I,J$, it follows from Lemma \ref{lemma:products} that $\mathcal I\mathcal J=\langle I\otimes J \rangle$ is again compact.
Finally, the product distributes over finite joins: $\mathcal I_1(\mathcal I_2+\mathcal I_3)=\mathcal I_1\mathcal I_2 +\mathcal I_1\mathcal I_3$.
\end{proof}

It follows in particular that $\mathcal R$ has an associated spectrum of prime elements, which by \cite{bks:ideal} is a spectral space in the sense of Hochster~\cite{hochster:prime}. In the next few subsections we (re)define the spectrum and (re)prove its spectrality in a more traditional way, using localization, as it makes clear the parallel to usual commutative rings and we will need localization in any case.

\subsection{The Zariski spectrum}

Let $\mathcal R$ be a $\mathcal G$-graded commutative 2-ring.

\begin{defi}
A proper homogeneous ideal $\mathcal I$ of~$\mathcal R$ is \emph{prime} if it satisfies the usual condition, that $r\circ s\in \mathcal I$ implies either $r\in \mathcal I$ or $s\in \mathcal I$.
The \emph{\textup(homogeneous\textup) spectrum of $\mathcal R$}, denoted  $\Spec \mathcal R$, is the set of all prime ideals of $\mathcal R$ endowed with the Zariski topology. 
Thus by definition the closed subsets are those of the form 
\[
V(\mathcal I):=\{ \mathfrak p\in \Spec \mathcal R\mid \mathcal I\subseteq \mathfrak p \}
\]
for some homogeneous ideal $\mathcal I$ of~$\mathcal R$.
\end{defi}

\begin{lemma}
\label{lemma:spectop}
We have the familiar computational rules:
\begin{enumerate}
\item $V(0)=\Spec \mathcal R$ and $V(\mathcal R)=\emptyset$.
\item $ V(\mathcal I) \cup V(\mathcal J) = V(\mathcal I \mathcal J) $.
\item $ \bigcap_\lambda V(\mathcal I_\lambda) = V\left(\sum_\lambda \mathcal I_\lambda \right)$.
\end{enumerate}
In particular the Zariski topology is indeed a topology.
Moreover, the sets
\[D_r := \{ \mathfrak p\in \Spec \mathcal R\mid r\not\in \mathfrak p \} 
\quad \quad (r\in \mathcal R)\]
provide a basis of open subsets. 
\end{lemma}

\begin{proof}
The computational rules are easily verified. 
We then obtain
\[ 
V(\mathcal I) 
= V\left( \sum_{r\in \mathcal I} \langle r\rangle \right) 
= \bigcap_{r\in \mathcal I}  V(\langle r\rangle)
\]
which shows that the sets $V(\langle r\rangle)= \Spec \mathcal R\smallsetminus D_r $ ($r\in \mathcal R$) are a basis of closed subsets, which is equivalent to the second claim.
\end{proof}

\begin{remark}
\label{remark:spec_fct}
It is immediate to verify that every morphism $F\colon \mathcal R\to \mathcal R'$ induces a continuous map $\Spec(F)\colon \Spec \mathcal R'\to \Spec \mathcal R$ by $\Spec(F)(\mathfrak p):= F^{-1}\mathfrak p$, and that $\Spec$ is functorial: $\Spec(\id_\mathcal R)=\id_{\Spec \mathcal R}$ and $\Spec( F\circ F' )=\Spec F' \circ \Spec F$.
\end{remark}

\begin{lemma}
\label{lemma:max_prime}
Every maximal homogeneous ideal is prime, and every proper homogeneous ideal is contained in a prime ideal.
\end{lemma}

\begin{proof}
The second claim follows from the first one by the usual application of Zorn's lemma. 
In order to prove the first claim, note that $\mathcal I$ is prime iff $\mathcal R/\mathcal I$ is a \emph{domain}, by which of course we mean that it has no nonzero divisors: if $rs=0$ then $ r=0$ or $s=0$. Also, $\mathcal I$ is maximal iff $\mathcal R/\mathcal I$ has no homogeneous ideals other than $0$ and itself, and they are distinct. 
But the latter implies the first: if $\mathcal I$ is maximal and if $s\in \mathcal R/\mathcal I$ is a nonzero element, then $\Ann_{\mathcal R/\mathcal I}(s)\neq \mathcal R/\mathcal I$ (otherwise $s=0$) so by hypothesis we must have $\Ann_{\mathcal R/\mathcal I}(s)=0$, showing that there exists no $r\in \mathcal R/\mathcal I$ with $rs=0$.
\end{proof}

\begin{cor}
\label{cor:nonempty}
The spectrum $\Spec \mathcal R$ is empty if and only if $\mathcal R\simeq \{0\}$.
\end{cor}

\begin{proof}
In view of Lemma \ref{lemma:max_prime} it suffices to see that every non-invertible map~$r$ in $\mathcal R$ is contained in some proper ideal, \ie, we have to show that in this case $\langle r\rangle$ is proper. 
If this were not the case, then $\id_\unit\in \langle r\rangle$. 
Therefore, by Proposition \ref{prop:principal}, we would be able to write $\id_\unit = urs $ and $\id_\unit= trv$ for some isomorphisms $u$ and~$v$.
But this would imply that $r$ has both a left and a right inverse and is therefore invertible, in contradiction with the hypothesis.
\end{proof}

\begin{prop}
\label{prop:qcpt}
The topological space $\Spec \mathcal R$ is quasi-compact.
\end{prop}

\begin{proof}
Consider a cover $\Spec \mathcal R= \bigcup_{\lambda\in \Lambda} U_\lambda $ by open subsets, which by Lemma \ref{lemma:spectop} we may assume of the form $U_\lambda=D_{r_\lambda}$. Thus $\emptyset = \bigcap_\lambda V(\langle r_\lambda \rangle)=V(\sum_\lambda \langle r_\lambda \rangle)$.
This means that the ideal $\mathcal I := \sum_\lambda \langle r_\lambda \rangle $ is equal to $\mathcal R$ (otherwise there would be some prime containing it, by Lemma \ref{lemma:max_prime}).
Hence $\id_\unit\in \mathcal I $, and it follows from the explicit descriptions of principal ideals and sums  (Proposition \ref{prop:principal} and Lemma \ref{lemma:unions}) that there exist finitely many indices $\lambda_1,\ldots,\lambda_n$ and maps  $s_1,\ldots,s_n$ and $t_1,\ldots, t_n$ such that 
$\id_\unit=  s_1r_{\lambda_1} t_1 +\ldots+   s_nr_{\lambda_n} t_n$.
Therefore $\id_\unit \in \langle r_{\lambda_1}\rangle +\ldots + \langle r_{\lambda_n}\rangle$, that is 
$\mathcal R= \langle r_{\lambda_1}\rangle +\ldots + \langle r_{\lambda_n}\rangle $, that is 
$\Spec \mathcal R = D_{r_{\lambda_1}} \cup \cdots \cup D_{r_{\lambda_n}}$.
\end{proof}

\subsection{Multiplicative systems and localization}
\label{subsec:loc}
We define homogeneous multiplicative systems in a graded commutative 2-ring in the natural way and show that they satisfy a two-sided calculus of fractions. 
In particular this allows us to localize~$\mathcal R$ at a prime ideal.

\begin{defi}
\label{defi:mult_set}
A family $S\subseteq \Mor \mathcal R$ of morphisms of~$\mathcal R$ is called a \emph{homogeneous multiplicative system} if it contains all isomorphisms and is closed under taking composites and translates.
\end{defi}

\begin{remark}
\label{rem:dual_multiplicative}
If $S$ is a homogeneous multiplicative system in $\mathcal R$ and $r$ is some morphism of~$\mathcal R$, then it follows from Example \ref{ex:dual_translate} that $r \in S$ iff $r^\vee \in S$.
\end{remark}

\begin{example}
\label{ex:mult_r}
For any $r\in \mathcal R$, we denote by $S_r$ the smallest homogeneous multiplicative system of $\mathcal R$ containing~$r$. 
Note that $S_r$ consists precisely of all finite composites of twists of~$r$ and isomorphisms, \ie, all finite compositions of translates of~$r$.
\end{example}

\begin{example} 
\label{ex:mult_prime}
The complement $S_{\mathfrak p}:=\Mor \mathcal R\smallsetminus \mathfrak p$ of every homogeneous prime ideal~$\mathfrak p\in \Spec \mathcal R$ is a homogeneous multiplicative system of~$\mathcal R$.
\end{example}


\begin{prop}
\label{prop:fractions}
Every homogeneous multiplicative system $S$ in a graded commutative 2-ring~$\mathcal R$ satisfies both a left and a right calculus of fractions.
\end{prop}

\begin{proof}
We are only going to prove that $S$ satisfies a calculus of left fractions, because the proof for right fractions is dual (using the duality $(-)^\vee : \mathcal R\simeq \mathcal R^\op$, which stabilizes~$S$ by Remark~\ref{rem:dual_multiplicative}). 
Since by definition $S$ contains the identities of~$\mathcal R$ and is closed under composition, it remains to verify the following two assertions:
\begin{enumerate}
\item (Ore condition.) Given two morphisms $r$ and~$s$ with $s\in S$ as depicted,
\[
\xymatrix@1{
g\ar[r]^-r \ar[d]_{S \,\ni\, s} &  h \ar@{..>}[d]^{s' \,\in\, S} \\
 \ell \ar@{..>}[r]^-{r'} & m
}
\]
then there exist $s'\in S$ and $r'$ such that $s'r=r's$.

\item (Cancellation.) Given three morphisms $r$, $t$ an $s$ as depicted,
\[
\xymatrix@1{ 
m \ar@{..>}[r]^-{s'} & 
g \ar@/^/[r]^-{r} \ar@/_/[r]_-t &
 h \ar[r]^-s & 
  \ell
}
\]
with $s\in S$ and such that $sr=st$, then there exists a morphism $s'\in S$ such that $rs'=ts'$.
\end{enumerate} 
Since $\mathcal R$ is an $\mathsf{Ab}$-category, (2) may be conveniently reformulated as follows:
\begin{itemize}
\item[(2)] Given two morphisms $r\colon g\to h$ and $s\colon h\to \ell$ with $s\in S$ and $sr=0$, then there is an $s'\in S$ with $rs'=0$.
\end{itemize}

\noindent To prove (1), consider the following commutative diagram:
\begin{equation} \label{dinat_mult}
\xymatrix{
 *+[F]{h}  &
  *+[F]{g} \ar[r]^-s \ar[l]_-{r} &
 *+[F]{  \ell }
    \\
 \unit  h \ar[u]^{\sim}_{\lambda} & \unit g\ar[l]_-{\unit r} \ar[u]_\lambda \ar[r]^-{\unit s} &
 \unit \ell \ar[u]^{\sim}_\lambda  \\
 hh^\vee h \ar[u]^\sim_{\varepsilon_{h^\vee}\, h} \ar[d]_\sim^{h\, \varepsilon_{h}} &
  hh^\vee g \ar[l]_-{hh^\vee r} \ar[u]_{\varepsilon_{h^\vee}\, g} \ar[r]^-{hh^\vee s} \ar[d]^{hr^\vee g} &
   hh^\vee \ell \ar[u]_{\varepsilon_{h^\vee}\,\ell}^\sim \ar[d]^{hr^\vee \ell} & \\
 h\unit & hg^\vee g \ar[l]^-\sim_-{h\,\varepsilon_{g}} \ar[r]^-{hg^\vee s} &
*+[F] { hg^\vee \ell }
}
\end{equation}
where the bottom-left square commutes by applying $h\otimes -$ to a dinaturality square \eqref{dinat} for~$r$, the two top squares by the naturality of~$\lambda$, and all remaining squares by the bifunctoriality of~$\otimes$.
Therefore it suffices to define $r'$ to be the vertical composite from $\ell$ to $hg^\vee \ell$ on the right hand side, and $s'\in S$ to be the left-and-bottom composite from $h$ to~$hg^\vee\ell$.

Let us prove~(2).
Given such $r$ and $s$, by Proposition~\ref{prop:pseudo_comm} there exists a commutative diagram (where the $\bullet$'s denote some unnamed, possibly different objects)
\[
\xymatrix{ 
g \ar@/_/[d]_u^\sim \ar[r]^-r &
 h \ar[r]^-s  &
   {\ell}  \\
\bullet \ar[r]^-{s''} &
  \bullet \ar[r]^-{r'} & \bullet   \ar@/_/[u]_v^\sim
}
\]
where $r'$ is some twist of~$r$ (say $r'=m\otimes r$) and $s''$ some twist of~$s$,  
and where $u$ and~$v$ are invertible.
Note that $s\in S$ implies $s''u\in S$, and that $sr=0$ implies
\[
r' \circ (s''u) = v^{-1}sr =0 \;.
\] 
By applying $m^\vee \otimes-$ to this vanishing composite, and using the naturality of $\lambda$ and~$\varepsilon$ and the functoriality of~$\otimes$, we obtain the next commutative diagram.
\[
\xymatrix{
m^\vee g \ar@/_5ex/[rrdd]_-{s' \,:= }  \ar@/^5ex/[rrr]^-0 \ar[rr]_-{m^\vee (s''u)}^{\in \, S} &&
 m^\vee mg \ar[d]_{\varepsilon g}^\sim \ar[r]_-{m^\vee r' } &
  m^\vee m h \ar[d]_{\varepsilon h}^\sim \\
&&
 \unit g \ar[d]_\lambda^\sim \ar[r]_-{\unit r} &
  \unit h \ar[d]_\lambda^\sim \\
&&
 g \ar[r]_-r & h  
}
\]
The composite map labeled $s'$ satisfies $s'\in S$ and $rs'=0$ completing the proof of (2) and the statement.
\end{proof}


\begin{cor}
Let $\mathcal R$ be a graded commutative 2-ring, let $S\subseteq \Mor \mathcal R$ be a homogeneous multiplicative system, and let $\loc : \mathcal R \to S^{-1}\mathcal R$ be the localization of $\mathcal R$ at~$S$.
Then $S^{-1}\mathcal R$ is a graded commutative 2-ring and for a unique symmetric tensor structure~$\otimes$ making the localization functor $\loc$ symmetric monoidal and making the square 
\[
\xymatrix{
\mathcal R\times \mathcal R \ar[r]^-{\otimes} \ar[d]_{\loc \times \loc} & \mathcal R \ar[d]^{\loc} \\
S^{-1} \mathcal R\times S^{-1} \mathcal R \ar[r]^-{\otimes} & S^{-1}\mathcal R
}
\]
strictly commute.
\end{cor}

\begin{proof}
It is a straightforward verification using the calculus of fractions.
\end{proof}

\begin{notation}
In particular, by Examples \ref{ex:mult_r} and \ref{ex:mult_prime} we obtain
for every map $r\in \mathcal R$ and every homogeneous prime $\mathfrak p\in \Spec \mathcal R$
 localization morphisms of graded commutative 2-rings
\[
\loc_r\colon \mathcal R\longrightarrow S^{-1}_r \mathcal R =: \mathcal R_r
\quad \textrm{ and } \quad
\loc_{\mathfrak p} \colon  \mathcal R \longrightarrow S_{\mathfrak p}^{-1}\mathcal R=: \mathcal R_{\mathfrak p}
\;,
\]
 ``away from $r$'' and ``at~$\mathfrak p$'', respectively.
\end{notation}

\begin{remark}
\label{rem:loc_invertible}
Note that  (simply because $\mathfrak p$ is an ideal) the multiplicative system~$S_\mathfrak p$ is \emph{saturated}, that is, if we are given three composable maps 
\[
\xymatrix{
g \ar[r]^-r & {h \phantom{g}}\!\!\! \ar[r]^-s & {\ell \phantom{g}}\!\!\! \ar[r]^-t & m
}
\]
with $ts\in S_\mathfrak p$ and $sr \in S_\mathfrak p$, it must follow that $s\in S_\mathfrak p$.
Hence $S_\mathfrak p$ consists precisely of all the morphisms in $\mathcal R$ whose image in~$\mathcal R_{\mathfrak p}$ is invertible.

\end{remark}

\subsection{Applications to the spectrum}
We now apply the calculus of fractions for homogeneous multiplicative systems in order to prove properties of the spectrum.

\begin{prop}
\label{prop:loc_spec}
Given a homogeneous multiplicative subset $S\subset \mathcal R$, the localization morphism $\loc\colon \mathcal R\to S^{-1}\mathcal R$ induces a homeomorphism 
\[
\Spec (\loc) \colon \Spec S^{-1}\mathcal R \stackrel{\sim}{\longrightarrow} 
\{ \mathfrak p \mid \mathfrak p\cap S =\emptyset \} \;\subseteq\; \Spec \mathcal R
\]
onto its image.
In particular,
the morphisms $\loc_r$ and $\loc_\mathfrak p$ induce homeomorphisms 
\[
\Spec \mathcal R_r \cong D_r 
\quad \textrm{ and } \quad
\Spec \mathcal R_\mathfrak p \cong \{\mathfrak q\in \Spec \mathcal R \mid \mathfrak q \subseteq  \mathfrak p \}
\]
for every $r\in \mathcal R$ and every $\mathfrak p\in \Spec \mathcal R$.
\end{prop}

\begin{proof}
Let us first explain how to obtain the two special cases of the statement from the main claim. It follows immediately from the definition of a prime ideal~$\mathfrak p$ that $r\in \mathfrak p$ iff $s\in \mathfrak p$ for some $s\in S_r$; hence $\mathfrak p\cap S_r=\emptyset $ iff $r\notin \mathfrak p$. Thus $\Spec \mathcal R_r \cong D_r $ by the first part of the proposition. The last homeomorphism is proved similarly. 

Now let $S$ be an arbitrary homogeneous multiplicative subset. For each homogeneous ideal $\mathcal I \subseteq \mathcal R$ we consider the following collection of morphisms of $S^{-1}\mathcal R$ (expressed as left fractions):
\[
S^{-1}\mathcal I :=\{ s^{-1}r \mid s\in S , r\in \mathcal I \} \,.
\]
We now prove the proposition in a series of easy lemmas.

\begin{lemma}
\label{lemma:claim1}
The collection $S^{-1}\mathcal I$ forms a homogeneous ideal of $S^{-1}\mathcal R$.
\end{lemma}
\begin{proof}
Evidently $S^{-1}\mathcal I$ contains all zero maps and is closed under twists. 
Closure under left and right multiplication follows from the Ore condition for left fractions: if the fractions $s_1^{-1}r_1$ and $s_2^{-1}r_2$ represent two morphisms in $S^{-1}\mathcal R$ that are right, resp.\ left, composable with some map $s^{-1}r\in S^{-1}\mathcal I$, then we find in $\mathcal R$ a commutative diagram of the form
\[
\xymatrix@R=15pt@C=15pt{
&& \bullet && \bullet && \\
& \bullet \ar[ur]^{\tilde r} && \bullet \ar[ul]_{\tilde s_1}^\sim \ar[ur]^{\tilde r_2} &&
 \ar[ul]^\sim_{\tilde s} \bullet & \\
\bullet \ar[ur]^{r_1} &&
 \bullet \ar[ul]_{s_1}^\sim \ar[ur]^{r} &&
  \bullet  \ar[ul]_{s}^\sim \ar[ur]^{r_2} && \bullet \ar[ul]^\sim_{s_2}
}
\]
with $\tilde s_1, \tilde s\in S$. Since $\mathcal I$ is an ideal and $S$ is closed under composition, we deduce from the equations $(s^{-1}r)(s_1^{-1}r_1)= (\tilde s_1s)^{-1}(\tilde r r_1)$ and 
$(s_2^{-1}r_2)(s^{-1}r)= (\tilde s s_2)^{-1}(\tilde r_2 r)$
that both composites belong again to $S^{-1}\mathcal I$.
Next we prove closure under sums. Consider two summable fractions $s^{-1}_1r_1, s^{-1}_2r_2\in S^{-1}\mathcal I$. By applying Ore's condition again, we obtain in~$\mathcal R$ a diagram as follows,
\[
\xymatrix{ 
& \bullet \ar[d]^{\tilde s_1}_\sim & \\
\bullet \ar[ur]^{r_1} \ar[dr]_{r_2} & \bullet & \bullet \ar[ul]_{s_1} \ar[dl]^{s_2} \\
&\bullet \ar[u]_{\tilde s_2}&
}
\]
where the right half is commutative and (say) $\tilde s_1\in S$. 
By the very definition of the sum of morphisms in the localization $S^{-1}\mathcal R$, we have the equation
\[
s_1^{-1}r_1 + s_2^{-1}r_2 =  \underbrace{(\tilde s_1 s_1)}_{\in \, S}\! {}^{-1}\underbrace{(\tilde s_1r_1 + \tilde s_2r_2)}_{\in \, \mathcal I} 
\,,
\] 
from which we deduce that the sum belongs again to $S^{-1}\mathcal I$.
Thus $S^{-1}\mathcal I$ is a homogeneous ideal as claimed.  
\end{proof}

\begin{lemma}
\label{lemma:claim2}
 If $\mathfrak p\in \Spec \mathcal R$ is such that $\mathfrak p\cap S=\emptyset$, then $S^{-1}\mathfrak p\in \Spec S^{-1}\mathcal R$.
\end{lemma}
\begin{proof} For this, we may assume without loss of generality that $S$ is saturated, since $\mathcal R\smallsetminus \mathfrak p$ contains $S$ and is saturated (Remark~\ref{rem:loc_invertible}).  
Let us see that $S^{-1}\mathfrak p$ is proper. If not then we may write $\id_\unit = s^{-1}r \in S^{-1}\mathfrak p$, from which we deduce by saturation that~$r\in S$ as well, but this would contradict our hypothesis that $S\cap \mathfrak p=\emptyset$. 
Now assume that $(s^{-1}_1r_1)(s^{-1}_2r_2)\in S^{-1}\mathfrak p$ for two left fractions $s^{-1}_1r_1, s^{-1}_2r_2 \in S^{-1}\mathcal R$.
Thus $(s^{-1}_1r_1)(s^{-1}_2r_2)= s^{-1} r$ for some $s\in S$ and $r\in \mathfrak p$.
By Ore, we obtain in~$\mathcal R$ a commutative diagram
\[
\xymatrix@R=15pt@C=15pt{
&& \bullet && \\
& \bullet \ar[ur]^{\tilde r_1} && \bullet \ar[ul]_{\tilde s_2}^\sim & \\
\bullet \ar[ur]^{r_2} && \bullet \ar[ul]_{s_2}^\sim \ar[ur]^{r_1} && \bullet  \ar[ul]_{s_1}^\sim
}
\]
with $\tilde s_2\in S$ and therefore an equation $(s^{-1}_1r_1)(s^{-1}_2r_2)= (\tilde s_2s_1)^{-1}(\tilde r_1r_2)$. 
Hence $s^{-1} r$ and $(\tilde s_2s_1)^{-1}(\tilde r_1r_2)$ are equivalent fractions and thus admit a common amplification, \ie, there exists in~$\mathcal R$ a commutative diagram
\[
\xymatrix{ 
& \bullet \ar[d]^{t} & \\
\bullet \ar[ur]^{r} \ar[dr]_{\tilde r_1 r_2} & \bullet & \bullet \ar[ul]_{s} \ar[dl]^{\tilde s_2s_1} \\
&\bullet \ar[u]_{u}^\sim &
}
\]
with $u\in S$. In particular we see that $u(\tilde r_1r_2)= tr\in \mathfrak p$, but since $u\not\in \mathfrak p$ we must have that $\tilde r_1r_2\in \mathfrak p$, so either $\tilde r_1$ or $r_2$ must belong to~$\mathfrak p$.
In the latter case $s^{-1}_2r_2\in S^{-1}\mathfrak p$; in the former, the equation $\tilde r_1s_2 =\tilde s_2 r_1$ implies that $r_1\in \mathfrak p$ and thus $s^{-1}_1r_1\in S^{-1}\mathfrak p$.
This concludes the proof that $S^{-1}\mathfrak p$ is a prime ideal in $S^{-1}\mathcal R$. 
\end{proof}

\begin{lemma} 
\label{lemma:claim3}
The construction $S^{-1}$ is left inverse to $\loc^{-1}$, \ie, every homogeneous ideal $\mathcal J$ of $ S^{-1}\mathcal R$ has the form $\mathcal J= S^{-1}(\loc^{-1} \mathcal J)$.
\end{lemma}
\begin{proof}
The inclusion $S^{-1}(\loc^{-1}\mathcal J)\subseteq \mathcal J$ is obvious. 
On the other hand, if a fraction $s^{-1}r $ belongs to~$ \mathcal J$ then $\loc(r) = ss^{-1}r \in \mathcal J$ too, so that $s^{-1}r\in S^{-1}(\loc^{-1}\mathcal J)$. This proves the other inclusion $\mathcal J\subseteq S^{-1}(\loc^{-1}\mathcal J)$ and therewith the claim.
\end{proof}

\begin{lemma}
\label{lemma:claim4} The two maps $\loc^{-1}$ and  $S^{-1}$ induce mutually inverse bijections 
between $ \Spec S^{-1}\mathcal R $ and $ \{\mathfrak p\in \Spec \mathcal R\mid \mathfrak p\cap S= \emptyset \} $.
\end{lemma}
\begin{proof}
We have already seen (in Lemma \ref{lemma:claim2} and Remark \ref{remark:spec_fct}) that $S^{-1}$ and $\loc^{-1}$ restrict to prime ideals as described (for the latter, note that $\loc^{-1}(\mathfrak q)\cap S=\emptyset$ because every $\mathfrak q\in \Spec S^{-1}\mathcal R $ is a proper ideal). By Lemma \ref{lemma:claim3}, we have $S^{-1}(\loc^{-1}\mathfrak q)=\mathfrak q$ for all $\mathfrak q\in \Spec S^{-1}\mathcal R$, and the inclusion $ \mathfrak p \subseteq \loc^{-1}(S^{-1}\mathfrak p)$
is obvious for all $\mathfrak p\in \Spec \mathcal R$ with $\mathfrak p\cap S=\emptyset$.
To prove the reverse inclusion, let $r\in \mathcal R$ be such that $\loc(r)\in S^{-1}\mathfrak p$. 
Then, by the definition of $S^{-1}\mathfrak p$, there exists in~$\mathcal R$ a commutative diagram
\[
\xymatrix{ 
& \bullet \ar[d]^{v} & \\
\bullet \ar[ur]^{t} \ar[dr]_{r} & \bullet & \bullet \ar[ul]_{s}^<<<<\sim \ar@{=}[dl] \\
&\bullet \ar[u]_{u}^\sim &
}
\]
with $s,u\in S$ and $t\in \mathfrak p$, from which we see that $ur=vt\in \mathfrak p$.
Since $u\not\in \mathfrak p$ by hypothesis and $\mathfrak p$ is prime, we conclude that $r$ belongs to~$\mathfrak p$. Therefore $\loc^{-1}(S^{-1}\mathfrak p)\subseteq \mathfrak p$ as well.
\end{proof}

\begin{lemma}
\label{lemma:claim5}  For any left fraction $s^{-1}r\in S^{-1}\mathcal R$, the preimage of  $V(\langle s^{-1}r\rangle )$ under the map 
$S^{-1}\colon \{\mathfrak p\in \Spec \mathcal R\mid \mathfrak p\cap S=\emptyset \}\stackrel{\sim}{\to} \Spec S^{-1}\mathcal R$ is $V(\langle r\rangle )$.
\end{lemma}
\begin{proof}
Consider  some $\mathfrak p\in \Spec \mathcal R$ with $\mathfrak p\cap S=\emptyset$. Clearly, if $r\in \mathfrak p$, then $s^{-1} r\in S^{-1}\mathfrak p$. Conversely, if $s^{-1}r\in S^{-1}\mathfrak p$ then---by the easy argument already employed twice---we must have $r\in \mathfrak p$. Therefore $(S^{-1})^{-1}V(\langle s^{-1}r\rangle ) = V(\langle r\rangle)$, as claimed.
\end{proof}

Finally, in order to complete the proof of the proposition it suffices to note that the bijection in Lemma \ref{lemma:claim4} is a homeomorphism for the respective Zariski topologies: the map $\loc^{-1}$ was already seen to be continuous, and its inverse $S^{-1}$ is continuous by virtue of Lemma \ref{lemma:claim5}. So we are done.
\end{proof}

We record an easy but pleasant consequence of the proof: the operations of taking quotients and localizations commute with one another.

\begin{cor}
\label{cor:SvsI}
Let $\mathcal R$ be a graded commutative 2-ring, let  $\mathcal I$ be a homogeneous ideal of $\mathcal R$, and let $S$ be a homogeneous multiplicative system in~$\mathcal R$. Then there exists a unique isomorphism of graded commutative 2-rings
\[
S^{-1}\mathcal R /S^{-1}\mathcal I \stackrel{\sim}{\to} (S/\mathcal I)^{-1}(\mathcal R/\mathcal I)
\]
which is compatible with the localization and quotient morphisms, where $S^{-1}\mathcal I$ is the homogeneous ideal of Lemma \ref{lemma:claim1}, and where $S/\mathcal I$ is the homogeneous system generated in $\mathcal R/\mathcal I$ by the image of~$S$.
\end{cor}

\begin{proof}
The proof is obvious from the universal properties of quotients and localizations.
\end{proof}

\begin{prop}
\label{prop:2spectral}
The spectrum $\Spec \mathcal R$ of every graded commutative 2-ring~$\mathcal R$ is a spectral topological space, \ie: it is $T_0$,  quasi-compact, it has a basis of quasi-compact open subsets closed under finite intersections, and every irreducible closed subset has a unique generic point.
Moreover, we may take $\{D_r\mid r\in \mathcal R\}$ as a basis of quasi-compact opens.
\end{prop}

\begin{proof}
We have already proved in Proposition \ref{prop:qcpt} that the whole spectrum is quasi-compact.
Moreover, the basic open subsets $D_r$, $r\in \mathcal R$, of Lemma \ref{lemma:spectop} are also quasi-compact, because of the homeomorphisms $D_r\cong \Spec \mathcal R_r$ of Proposition~\ref{prop:loc_spec}. As in the case of usual rings it is clear that they are closed under finite intersections (indeed $D_r\cap D_s = D_{\tilde rs}$ for any translate $\tilde r$ of $r$ that is composable with~$s$).
Thus it only remains to prove the existence and uniqueness of generic points.
Uniqueness and the fact that the spectrum is $T_0$ are immediate from the definition of the Zariski topology, from which we see that the closure of a point has the form $\overline{\{\mathfrak p\}}=V(\mathfrak p)$; accordingly, if $\mathfrak p_1$ and $\mathfrak p_2$ have the same closure then they are contained in one another and hence equal.
For the existence, it suffices to show that every nonempty close subset $Z \subseteq \Spec \mathcal R$ contains a minimal point (with respect to inclusion). 
Writing $Z= V(\mathcal I)$, this is equivalent to showing that if $\mathcal I$ is proper then there is a minimal prime containing it (as $Z\cong \Spec \mathcal R/\mathcal I$ and by Corollary \ref{cor:nonempty}). 
This is a standard application of Zorn's lemma.
\end{proof}

\begin{thm}
\label{thm:spec}
For every morphism $F\colon \mathcal R\to \mathcal R'$ of graded commutative 2-rings, there is a spectral continuous map $\Spec F\colon \Spec \mathcal R'\to \Spec \mathcal R$ given by $\mathfrak p\mapsto F^{-1}\mathfrak p$.
This defines a contravariant functor, $\Spec$, from graded commutative 2-rings and their morphisms to spectral topological spaces and spectral continuous maps.
\end{thm}

\begin{proof}
In view of the last result it remains only to verify that $\Spec F$ is a spectral continuous map (\ie, that the preimage of a quasi-compact open is again a quasi-compact open). For this it suffices to notice that 
\begin{displaymath}
(\Spec F)^{-1}D_r= \{ \mathfrak p\in \Spec \mathcal{R}' \mid r\not\in F^{-1}\mathfrak p\}= \{ \mathfrak p\in \Spec \mathcal{R}' \mid Fr \not\in \mathfrak p\} = D_{Fr}
\end{displaymath}
for all $r\in \mathcal R$.
\end{proof}

\subsection{Localization of $\mathcal R$-algebras}
\label{subsec:loc_alg}

For our applications we will need  to localize not only 2-rings but also algebras over them, as we now explain. Thus we need to generalize Proposition \ref{prop:fractions} accordingly.
Let $\mathcal R$ be a graded commutative 2-ring. 

\begin{defi}
\label{defi:algebra}
An \emph{algebra over~$\mathcal R$} (or \emph{$\mathcal R$-algebra}) is a symmetric monoidal $\Z$-category $\mathcal A$ equipped with a additive symmetric monoidal functor $F\colon \mathcal R\to \mathcal A$. (Note that the objects of $\mathcal A$ are not required to be invertible.)
\end{defi}

\begin{notation}
Let $\mathcal R$ be a graded commutative 2-ring, let $S$ be a homogeneous multiplicative system in~$\mathcal R$, and let $F\colon \mathcal R\to \mathcal A$ be an $\mathcal R$-algebra. 
Write  $ S_\mathcal A $ for the smallest class of maps in~$\mathcal A$ containing~$FS$ and all isomorphisms of~$\mathcal A$ and which is closed under composition and twisting with objects of $\mathcal{A}$.
\end{notation}

\begin{thm}
\label{thm:fractions_general}
Let $\mathcal R$ be a graded commutative 2-ring, let $S$ be a homogeneous multiplicative system in~$\mathcal R$, and let $F\colon \mathcal R\to \mathcal A$ be an $\mathcal R$-algebra. 
Then $S_\mathcal A$ satisfies  in~$\mathcal A$ both a left and a right calculus of fractions.
\end{thm}

\begin{proof}
To begin with, notice that $S_\mathcal A$ would remain the same if we substitute $\mathcal R$ with the full subcategory on the replete closure of its image $F\mathcal R$, and $S$ with the homogeneous multiplicative system generated by the images of maps in~$S$. Thus without loss of generality we may assume that $\mathcal R$ is a full replete subcategory of~$\mathcal A$.

Let us verify that $S_\mathcal A$ satisfies a calculus of left fractions (the proof for right fractions is dual and will be omitted). 
Since by definition $S_\mathcal A$ contains all identity maps and is closed under composition, it remains to verify conditions (1) and~(2) as in the proof of Proposition~\ref{prop:fractions}.

We see that the set $S_\mathcal A$ consists precisely of finite composites of maps in~$\mathcal A$ which are either invertible or belong to~$\{x\otimes s \mid s\in S, x\in \obj \mathcal A\}$ (or alternatively, to $\{s\otimes x\mid s\in S, x\in \obj \mathcal A\}$).
Thus to prove~(1) it evidently suffices to consider diagrams $\smash{\bullet \leftarrow \bullet \rightarrow \bullet}$ where the map $(\bullet \leftarrow \bullet)\in S_\mathcal A$ is a twist in $\mathcal A$ of a map of~$S$. 
Accordingly, assume we are given two maps
\[
\xymatrix{
{\phantom{g} xh} & {\!\!\! \phantom{h} xg} \ar[l]_-{xs} \ar[r]^-r & {\!\!\! \phantom{h} y} 
}
\]
for some $s\in S\subseteq \mathcal R$ and $x,y\in \obj \mathcal A$. 
In order to complete them to a square as required, draw the following commutative diagram, which is similar to~\eqref{dinat_mult}.
\begin{equation} 
\label{dinat_mult_general}
\xymatrix@R=8pt@C=8pt{
&
 *+[F]{xh} &
  &
    *+[F]{xg} \ar[ll]_-{xs} \ar[rr]^-{r} &&
     *+[F]{y} \\
&&&&& \\
&
 \unit xh \ar[uu]^\sim_{\lambda_{xh}} &&
  \unit xg \ar[ll]_-{\unit xs} \ar[uu]^{\lambda_{xg}} \ar[rr]^-{\unit r} &&
   \unit y \ar[uu]_-\sim^-{\lambda_y} \\
&&&&& \\
&
 hh^\vee xh \ar[uu]_{\varepsilon_{h^\vee} xh}^\sim \ar[dl]^<<<{\gamma h}_-\sim &&
  hh^\vee xg \ar[ddrr]|{hs^\vee r} \ar[uu]^{\varepsilon_{h^\vee} xg} \ar[ll]_-{hh^\vee xs} \ar[rr]^-{hh^\vee r} \ar[dl]_<<<<<<{\gamma g} \ar[dd]|{hs^\vee xg \;\; } &&
   hh^\vee y \ar[uu]_\sim^{\varepsilon_{h^\vee} y} \ar[dd]^{hs^\vee y} \\
xhh^\vee h \ar[dd]_\sim^{xh \varepsilon_{h}} &&
 xhh^\vee g \ar[ll]^-{xhh^\vee s} \ar[dd]_{xh s^\vee g} &&& \\
&&& hg^\vee xg \ar[rr]_{hg^\vee r} \ar[dl]^\sim_{\gamma g} &&
 *+[F]{hg^\vee y} \\
xh\unit && xhg^\vee g \ar[ll]^-\sim_-{xh \varepsilon_{g}} &&&
}
\end{equation}
The top two squares commute by naturality of~$\lambda$; the middle-bottom skew one by the naturality of~$\gamma$; the bottom-left one by applying $xh\otimes-$ to a dinaturality square for~$s$, and all the remaining squares by functoriality of the tensor. That each map marked $\sim$ is an isomorphism is either clear or follows from the fact that both $g$ and $h$ are invertible, which is the case since $s\in S$.
If we compose maps between the objects in boxes, the outer frame of \eqref{dinat_mult_general} becomes a commutative square
\[
\xymatrix{
xg \ar[d]_{xs} \ar[r]^{r} &
 y \ar[d]^{s'} \\
xh \ar[r] & hg^\vee y
}
\]
with $s'\in S_\mathcal A$ (by Remark \ref{rem:dual_multiplicative}). This proves~(1).
It remains to prove condition~(2) for $S_\mathcal A$, which reads as follows:
\begin{itemize}
\item[(2)] Given two morphisms $r\colon x\to y$ and $s\colon y\to z$ such that $s\in S_\mathcal A$ and $sr=0$, then there is an $s'\in S_{\mathcal A}$ such that $rs'=0$.
\end{itemize}
 
We first prove the following special case of~(2).

\begin{lemma}
\label{lemma:ind_reduction}
Consider two composable maps $r\colon x \to g\otimes y$ and $s\otimes y\colon g\otimes y\to h \otimes y$ such that  $(s\colon g\to h)\in S$ and $(s\otimes y)r=0$. Then there is an $s'\in S_\mathcal A$ with $rs'=0$. 
\end{lemma}

\begin{proof}
Given such $r$ and~$s$, build the following commutative diagram, whose similarity with \eqref{dinat_mult_general} will not be missed:
\begin{equation*}
\xymatrix@R=8pt@C=8pt{
x \ar[rr]^-r \ar@/^4ex/[rrrr]^-0 && gy \ar[rr]^-{sy} && hy & \\
&&&&& \\
x\unit \ar[uu]^\sim_\rho \ar[rr]^-{r\unit} &&
 gy\unit \ar[uu]_\rho \ar[rr]^-{sy\unit} &&
  hy\unit \ar[uu]_\sim^\rho & \\
&&&&& \\
*+[F]{xh^\vee h} \ar[uu]^\sim_\varepsilon \ar[rr]^-{rh^\vee h} \ar[dd]_{xs^\vee h} &&
 gyh^\vee h \ar[uu]_\varepsilon \ar[rr]^-{syh^\vee h} \ar[rd]_-{g\gamma} \ar[dd]_{gys^\vee h} &&
  hyh^\vee h \ar[uu]_\sim^\varepsilon \ar[rd]^\sim_<<{h\gamma} & \\
&&& gh^\vee hy \ar[rr]_-{sh^\vee hy} \ar[dd]^{gs^\vee hy} &&
 hh^\vee hy \ar[dd]_{\varepsilon hy}^\sim \\
*+[F]{xg^\vee h} \ar[rr]_-{rg^\vee h} &&
 *+[F]{gyg^\vee h} \ar[dr]_\sim^{g\gamma} &&& \\
&&& gg^\vee hy \ar[rr]_-\sim^-{\varepsilon hy} && \unit hy
}
\end{equation*}
Here, the two top squares commute by the naturality of~$\rho$, the middle-bottom skew one by the naturality of~$\gamma$, the bottom-right one by applying $-\otimes hy$ to a dinaturality square for~$s$, and all the other squares  by the functoriality of the tensor. Again the maps marked~$\sim$ are all isomorphisms since both $g$ and $h$ are $\otimes$-invertible, which is the case since $s\in S$.
The outer frame of the diagram tells us that the two composable arrows between the framed objects, namely
\[
\xymatrix{
xh^\vee h \ar[r]^-{xs^\vee h} & xg^\vee h \ar[r]^-{rg^\vee h} & gyg^\vee h
} \;,
\]
compose to zero. Note also that $s'':= xs^\vee h \in S_\mathcal A$, since $s\in S$.
Now it suffices to untwist this composition by the invertible object~$g^\vee h$.
More precisely, the following commutative diagram
\begin{equation*}
\xymatrix{
xh^\vee (g^\vee h)^\vee \ar[rr]_-{xs^\vee h(g^\vee h)^\vee} \ar@/^4ex/[rrrr]^-{0} \ar@/_5ex/[ddrr]_{s'\,:=} &&
 xg^\vee h (g^\vee h)^\vee \ar[rr]_-{rg^\vee h (g^\vee h)^\vee} \ar[d]_{x\varepsilon} &&
  gyg^\vee h (g^\vee h)^\vee \ar[d]_{gy\varepsilon}^\sim  \\
&& x \unit \ar[rr]_-{r\unit} \ar[d]_\rho && gy\unit \ar[d]_\rho^\sim \\
&& x \ar[rr]_-r && gy
}
\end{equation*}
defines a morphism $s'\in S_\mathcal A$ such that $r s' = 0$, as required. 
\end{proof}

\begin{lemma}
\label{lemma:S_A}
Every $(s\colon y\to z)\in S_\mathcal A$ is a finite composition of the form
\[
\xymatrix{
y \ar[r]^-{u_0}_-\sim &
 g_1a \ar[r]^-{s_1a} &
  h_1 a \ar[r]^-{u_1}_-\sim &
   \cdots 
    g_ia \ar[r]^-{s_i a} & 
     h_i a \ar[r]^-{u_i}_-\sim &
      g_{i+1} a \cdots \ar[r]^-{ s_na} &
       h_n a \ar[r]^-{u_n}_\sim & 
       z
}
\]
where $a$ is some object of~$\mathcal A$, each $s_i \colon g_i\to h_i$ belongs to $S$ (so in particular $g_i ,h_i$ are invertible objects in~$\mathcal R$), and each $u_i$ is an isomorphism.
\end{lemma}

\begin{proof}
If $(s\colon y\to z) \in S_\mathcal A$, then $s$ must be a finite composite of the form
\[
\xymatrix{
y \ar[r]^-{v_0}_-\sim &
 \tilde g_1a_1 \ar[r]^-{\tilde s_1a_1} &
  \tilde h_1 a_1 \ar[r]^-{v_1}_-\sim &
   \cdots 
    \tilde g_ia_i \ar[r]^-{\tilde s_i a_i} & 
     \tilde h_i a_i \ar[r]^-{v_i}_-\sim &
       \cdots \ar[r]^-{ \tilde s_na_n} &
       \tilde h_n a_n \ar[r]^-{v_n}_-\sim & 
       z
}
\]
for some isomorphisms $v_0, \ldots, v_n$, some $\tilde s_1,\ldots , \tilde s_n\in S$, and some objects $a_1,\ldots a_n$.
(To see this, it suffices to note that twists in $\mathcal A$ preserve isomorphisms, and that every left twist $x\otimes r$ of an arrow~$r$ may be turned into a right twist composed with two isomorphisms, namely $\gamma (r\otimes x) \gamma$.)
We deduce in particular from the isomorphisms $v_i\colon \tilde h_ia_i \cong\tilde  g_{i+1}a_{i+1}$ that there exist isomorphisms $a_{i+1}\cong \tilde g_{i+1}^\vee \tilde  h_i a_i $ and therefore, recursively, that there exist isomorphisms
\[
w_i \colon a_i \stackrel{\sim}{\longrightarrow} (\underbrace{ \tilde g_i^\vee\tilde  h_{i-1} \tilde g_{i-1}^\vee\tilde  h_{i-2} \cdots \tilde g^\vee_2\tilde  h_1 }_{=:\; \ell_i}) a_1
\]
for $i=1,\ldots, n$ (use $\ell_1=\unit$ when $i=1$). 
Setting $a:=a_1$, as well as defining $\ell_i$ as above and $u_i$ and $s_i$ as displayed below, we obtain the following commutative diagram, where the top row is the given map~$s$.
\[
\xymatrix@C=22pt{
y \ar[r]^-{v_0}_-\sim  &
 \tilde g_1a_1 \ar[r]^-{\tilde s_1 a_1} \ar[d]^{\tilde  g_1w_1}_\sim &
  \tilde h_1 a_1 \ar[r]^-{v_1}_-\sim \ar[d]^{\tilde  h_1w_1}_\sim &
   \cdots 
    \tilde g_i a_i \ar[r]^-{\tilde s_i a_i} & 
     \tilde h_i a_i \ar[r]^-{v_i}_-\sim &
       \cdots \ar[r]^-{ \tilde s_n a_n} &
       \tilde h_n a_n \ar[r]^-{v_n}_\sim & 
       z \\
y \ar@{=}[u] \ar@{..>}[r]^-{u_0} &
 \tilde g_1\ell_1 a  \ar@{..>}[r]^-{s_1 a} &
 \tilde  h_1\ell_1 a  \ar@{..>}[r]^-{u_1} &
  \cdots \tilde  g_i \ell_i a \ar@{<-}[u]_{\tilde g_i w_i}^\sim \ar@{..>}[r]^-{s_i a} &
  \tilde h_i\ell_i a \ar@{<-}[u]_{\tilde h_i w_i}^\sim \ar@{..>}[r]^-{u_i} &
   \cdots \ar@{..>}[r]^-{s_n a} &
    \tilde h_n\ell_n a  \ar@{<-}[u]_{\tilde h_n w_i}^\sim \ar@{..>}[r]^-{u_n} &
     z \ar@{=}[u] 
}
\]
Note that each $s_i:= \tilde s_i \otimes \ell_i $ belongs to~$S$ (because $\tilde s_i\in S$ and $\ell_i\in \mathcal R$) 
and that each $u_i$ is an isomorphism. 
If we further set $g_i:=\tilde g_i \otimes \ell_i$ and $h_i:= \tilde h_i\otimes \ell_i$ we see from the bottom row of the diagram that $s$ has the claimed form.
\end{proof}

We are now ready to verify property (2) for general morphisms of~$S_\mathcal A$. 
The proof is an easy recursion (probably best drawn on the blackboard).
Let $r,s$ be as in~(2). 
Since $s\in S_\mathcal A$, by Lemma \ref{lemma:S_A} we have 
\[
s = u_n (s_n \otimes a) u_{n-1} (s_{n-1}\otimes a) \cdots u_1 (s_1\otimes a) u_0 
\colon y \longrightarrow z
\]
for some object $a\in \mathcal A$, some isomorphisms~$u_i$, and some $s_i\colon g_i\to h_i$ in~$S$.
By hypothesis we have $sr=0$. 
Since $u_n$ is an isomorphism, this implies
\begin{align*}
0 &=  \big( (s_n \otimes a) u_{n-1} (s_{n-1}\otimes a) \cdots u_1 (s_1\otimes a) u_0 \big) \circ r \\
 &= (s_n \otimes a)\circ \big( \underbrace{u_{n-1} (s_{n-1}\otimes a) \cdots (s_1\otimes a) u_0r}_{=: \; r_n } \big) \,.
\end{align*}
Since $s_n\in S$, we can apply Lemma \ref{lemma:ind_reduction} to deduce the existence of some $s'_n\in S_{\mathcal A}$ with $r_ns_n'=0$. 
Since $u_{n-1}$ is an isomorphism, we actually have
\begin{align*}
0 &=   \big( (s_{n-1} \otimes a) u_{n-2} (s_{n-2}\otimes a) \cdots u_1 (s_1\otimes a) u_0 r \big) \circ s'_n \\
 &=  (s_{n-1} \otimes a)\circ \big(\underbrace{ u_{n-2} (s_{n-2} \otimes a)u_{n-3} (s_{n-3}\otimes a) \cdots (s_1\otimes a)u_0rs'_n}_{=: \; r_{n-1} } \big)
\end{align*}
and we may now iterate: by applying the same argument $n-1$ more times we successively produce morphisms $s'_{n-1},s'_{n-2},\ldots, s'_1$ such that  the composition 
$s':= s_n's'_{n-1}s'_{n-2} \ldots s'_1 $ belongs to 
$S_{\mathcal A}$ and satisfies $rs'=0$, as required.

This concludes the proof of Theorem~\ref{thm:fractions_general}.
\end{proof}

We next compare the localization of 2-rings with that of their algebras.

\begin{lemma}
\label{lemma:S_AvsA}
Let $F\colon \mathcal R\to \mathcal A$ be an $\mathcal R$-algebra, let $S\subseteq \Mor \mathcal R$ be a homogeneous multiplicative system, and let $S_\mathcal A\subseteq \Mor \mathcal A$ be its extension to~$\mathcal A$. 
Assume that $F$ is a full functor whose image is replete.
Then if $s\colon x\to y$ is a morphisms of $S_\mathcal A$ such that either $x$ or $y$ belongs to~$F\mathcal R$, we must have $s\in FS$.
\end{lemma}

\begin{proof}
The statement only concerns the images of $\mathcal R$ and $S$ in~$\mathcal A$, hence we may assume $F$ is the inclusion of a full and replete $\otimes$-subcategory $\mathcal R$ of invertible objects of~$\mathcal A$. 
Let $(s \colon x\to g)\in S_\mathcal A$ with $g\in \mathcal R$ (the proof for the other case is dual and is omitted). 
By Lemma~\ref{lemma:S_A} the morphism~$s$  is equal to a composite
\[
\xymatrix{
x \ar[r]^-{u_0}_-\sim &
 g_1a \ar[r]^-{s_1a} &
  h_1 a \ar[r]^-{u_1}_-\sim &
   \cdots 
    g_ia \ar[r]^-{s_i a} & 
     h_i a \ar[r]^-{u_i}_-\sim &
      g_{i+1} a \cdots \ar[r]^-{ s_na} &
       h_n a \ar[r]^-{u_n}_\sim & 
       g
}
\]
where the maps $s_i$ belong to~$S$, the maps $u_i$ are isomorphisms, and the objects $g_i,h_i$ belong to~$\mathcal R$.
We see immediately that $a \cong h_n^{-1}g$ lies in $\mathcal R$, and consequently so does $x\cong g_1a$.
Also, since $a$ is in $\mathcal R$ each map $s_ia \colon g_ia\to h_ia$ belongs to~$S$. 
By the fullness of~$\mathcal R$, the isomorphisms $u_i$ belong to $\mathcal R$ and therefore to~$S$. Hence the composition~$s$ belongs to~$S$, proving the lemma.
\end{proof}

\begin{prop}
\label{prop:comparison_mult_sys}
Let $F\colon \mathcal R\to \mathcal A $ be an $\mathcal R$-algebra and let $S\subseteq \mathcal R$ be a homogeneous multiplicative system. 
Then the unique canonical $\otimes$-functor 
$\overline F$
which makes the following square commute 
\[
\xymatrix{ \mathcal R \ar[d] \ar[r]^-F & \mathcal A \ar[d] \\
S^{-1}\mathcal R \ar[r]^-{\overline F} &
S^{-1}_\mathcal A \mathcal A
 }
\]
is full (fully faithful) if $F$ is full (fully faithful).
\end{prop}

\begin{proof}
We first assume that $F$ is fully faithful. Note that in this case we may further assume that $F$ is the inclusion of a full replete subcategory of~$\mathcal A$, the multiplicative systems arising from $F\mathcal{R}$ and its replete closure being identical. Let
\[
\xymatrix{
g & \ar[l]_-s^-{\sim} x \ar[r]^-f & h
}
\]
be a right fraction representing a morphism in $S^{-1}_\mathcal A\mathcal A$ such that $g,h\in \mathcal R$.
Since $g\in \mathcal R$, Lemma~\ref{lemma:S_AvsA} says that $s$ belongs to~$S$, showing that the fraction $fs^{-1}$ defines a morphism $g\to h$ in $S^{-1}\mathcal R$ as well. This proves that the functor $\overline F$ is full.
Next consider a fraction
\[
\xymatrix{
g & \ar[l]_-t^-\sim \ell \ar[r]^-r & h
}
\]
(with $t\in S$) representing a morphism of $S^{-1}\mathcal R$ which is mapped to zero in $S_\mathcal 
A^{-1}\mathcal A$.
The latter means that there exists in~$\mathcal A$ a commutative diagram
\begin{equation}
\label{eq:zeros}
\xymatrix@R=5pt{
& \ell \ar[dl]_t \ar[dr]^r & \\
g && h \\
& x \ar[ul]^-s \ar[ur]_0 \ar[uu] &
}
\end{equation}
for some $s\in S_\mathcal A$. By Lemma \ref{lemma:S_AvsA} once again, we must have $s\in S$.
But then \eqref{eq:zeros} means precisely that $rt^{-1}=0$ in $S^{-1}\mathcal R$. So $\overline F$ is fully faithful.

The more general case, where $F$ is assumed to be full but possibly not faithful, can be reduced to the previous one as follows. 
By factoring $F$ through its image, localization induces the commutative diagram
\[
\xymatrix{ 
\mathcal R \ar[d] \ar@/^3ex/[rr]^-F \ar[r]_-{\textrm{full}} &
 F\mathcal R \ar[r]_-{\textrm{faithful}} \ar[d] &
  \mathcal A \ar[d] \\
S^{-1}\mathcal R \ar@/_3ex/[rr]_-{\overline F} \ar[r] &
 (FS)^{-1}F\mathcal R \ar[r] &
S^{-1}_\mathcal A \mathcal A
 }
\]
(here $FS$ denotes the multiplicative system in $F\mathcal R$ generated by~$S$).
By Corollary \ref{cor:SvsI} (with $\mathcal I=\ker(F)$) we have the identification $(FS)^{-1} F\mathcal R \cong S^{-1}\mathcal R/S^{-1}\ker(F)$, from which we see that $S^{-1}\mathcal R\to (FS)^{-1}F\mathcal R$ is full.
Thus it remains only to verify that the functor $(FS)^{-1}F\mathcal R\to S^{-1}_\mathcal A\mathcal A$ is full, and this follows from what we have already proved.
\end{proof}

\section{Generalized comparison maps}
\label{sec:comparison}

\subsection{Central 2-rings of tensor triangulated categories}
\label{subsec:central2rings}

From now on we work with an essentially small tensor triangulated category~$\mathcal T$; thus $\mathcal T$ is essentially small, triangulated, and equipped with a symmetric tensor structure $(\mathcal T,\otimes,\unit,\alpha,\lambda,\rho, \gamma)$ such that $x\otimes-$ (and thus $-\otimes x$) preserves exact triangles for each object $x\in \mathcal T$.

\begin{defi}
By a \emph{central 2-ring of~$\mathcal T$} we mean any full tensor subcategory $\mathcal R$ of invertible objects of $\mathcal T$ which is closed under taking duals.
Thus every central 2-ring of $\mathcal T$ is a graded commutative 2-ring, as studied in the previous section.
\end{defi}

\begin{examples}
\label{ex:central2rings}
At one extreme we find $\mathcal R=\{\unit\}$, that is, the commutative endomorphism ring of the tensor unit, $\End_\mathcal T(\unit)$. In Balmer's notation this is the \emph{central ring} $\mathrm R_\mathcal T$ of~$\mathcal T$. At the other extreme we may choose $\mathcal R$ to be the full subcategory of all invertible objects in $\mathcal T$, which deserves the name of \emph{total central 2-ring of $\mathcal T$}, written $\mathrm R^\tot_\mathcal T$. 
Between $\mathrm R_\mathcal T$ and $\mathrm R^\tot_\mathcal T$ we find a poset of central 2-rings, ordered by inclusion, which in fact is a lattice with meet $\mathcal R\wedge \mathcal R'=\mathcal R\cap \mathcal R'$ and join $\mathcal R\vee \mathcal R' =\bigcap \{\mathcal R'' \mid \mathcal R\cup \mathcal R'\subseteq \mathcal R''\}$. 
Every inclusion $\mathcal R\hookrightarrow \mathcal R'$ is a morphism of graded commutative 2-rings and so  it induces a continuous map $\Spec \mathcal R'\to \Spec \mathcal R$.
\end{examples}

The next theorem is the key point in allowing us to relate the geometry of a $\otimes$-triangulated category to that of its central rings.

\begin{thm}
\label{thm:local_Rtot}
If $\mathcal T$ is a local $\otimes$-triangulated category (\ie, if its spectrum has a unique closed point~\cite{balmer:spec3}), then every central 2-ring $\mathcal R$ of $\mathcal T$ is local as a graded commutative 2-ring, \ie, it has a unique maximal homogeneous ideal. 
Moreover, this maximal ideal consists precisely of the non-invertible arrows of~$\mathcal R$.
\end{thm}

The proof is essentially the same as that of \cite{balmer:spec3}*{Theorem 4.5}, but a few slight adjustments are required. 
We need a couple of lemmas concerning tensor nilpotent morphisms.

\begin{lemma}
\label{lemma:nilp_iso}
Let $a\colon g\otimes x\to h\otimes x$ be a morphism  in a $\Z$-linear symmetric $\otimes$-category,
where $g$ and $h$ are two $\otimes$-invertible objects. 
If $a$ is both $\otimes$-nilpotent and an isomorphism, then $x$ is $\otimes$-nilpotent.
\end{lemma}

\begin{proof}
Suppose that $a^{\otimes n} = 0$. Then since for any $i\geq 1$ the map $a^{\otimes i}$ is an isomorphism we deduce that the object $g^{\otimes n} \otimes x^{\otimes n}$ is isomorphic to zero. Since $g$ is invertible it follows that $x^{\otimes n} \cong 0$ as claimed.
\end{proof}

\begin{lemma}
\label{lemma:nilp_sum}
Let $a,b\colon x\to y$ be two parallel maps in a $\Z$-linear symmetric $\otimes$-category.
If both $a$ and $b$ are $\otimes$-nilpotent, then so is $a+b$.
\end{lemma}

\begin{proof}
Since the $\otimes$-product is $\Z$-linear in each variable we can write $(a+b)^{\otimes n}$ as a sum of morphisms of the form
\begin{displaymath}
u_i (a^{\otimes n-i} \otimes b^{\otimes i}) v_i
\end{displaymath}
where $u_i$ and $v_i$ are some composites (depending on~$i\in \{0,\ldots,n\}$) of instances of $\alpha$, $\gamma$ and identity arrows tensored with each other. Since both $a$ and $b$ are $\otimes$-nilpotent we see that $(a+b)^{\otimes n}$ is zero for $n$ chosen sufficiently large i.e., $a+b$ is $\otimes$-nilpotent as claimed.
\end{proof}

\begin{proof}[Proof of Theorem \ref{thm:local_Rtot}]
By Corollary \ref{cor:composite_iso}, if we can show that the sum of any two non-invertible (parallel) maps is again non-invertible, then the collection of non-invertible maps in $\mathcal R$ is a homogeneous ideal which will necessarily be the unique maximal one, and we would be done. 

Thus assume that $r+s: g \to h$ is invertible; we must prove that either $r$ or $s$ is also invertible. 
To this end, consider the following morphism of $\mathcal T$:
\[
t := (r+s)\otimes \id \otimes \id \; \colon \;
g \otimes \cone(r) \otimes \cone (s) \longrightarrow
h \otimes \cone(r) \otimes \cone (s)
\;.
\]
Since $r+s$ is invertible, so is~$t$. We claim that $t$ is also $\otimes$-nilpotent. 
By Lemma \ref{lemma:nilp_sum} it suffices to show that both $r \otimes \id_{\cone(r)} \otimes \id_{\cone(s)}$ and $s \otimes \id_{\cone(r)} \otimes \id_{\cone(s)}$ are $\otimes$-nilpotent, and clearly it suffices to show that $r \otimes \id_{\cone(r)}$ and $s \otimes \id_{\cone(s)}$ are $\otimes$-nilpotent.
Since $g$ and $h$ are invertible objects of~$\mathcal T$, this follows immediately from \cite{balmer:spec3}*{Proposition 2.13}.
Thus $t$ is both $\otimes$-nilpotent and invertible, and by Lemma \ref{lemma:nilp_iso} we must have that $\cone(r)\otimes \cone(s)$ is tensor nilpotent and hence zero, because -- as a $\otimes$-product of cones of maps between invertible objects -- it is dualizable.
But $\mathcal T$ is local by assumption, so that $\cone(r)\otimes \cone(s)\cong 0$ implies that either $\cone(r)$ or $\cone(s)$ is $\otimes$-nilpotent. As we have already noted these cones are dualizable so either $\cone(r)\simeq 0$ or $\cone(s)\simeq 0$.
We conclude that either $r$ or $s$ is already an isomorphism.
\end{proof}

\subsection{Central localization}

The following theorem is a generalization of Balmer's procedure of central localization (see \cite{balmer:spec3}*{\S3}).

\begin{thm}
\label{thm:central_loc}
Let $\mathcal R$ be a central 2-ring of a $\otimes$-triangulated category~$\mathcal T$, and let $S$ be a homogeneous multiplicative system in~$\mathcal R$. 
Then localization induces a canonical isomorphism
\[
\xymatrix{
\mathcal T \ar[r]^-q \ar[d]_\loc & \mathcal T/\mathcal J_S \\
S_\mathcal T^{-1} \mathcal T \ar[ur]_\cong &
}
\]
between $\mathcal T$ localized at $S$ as an $\mathcal R$-algebra (see Theorem \ref{thm:fractions_general}) and the Verdier quotient of~$\mathcal T$ by the thick $\otimes$-ideal $\mathcal J_S:=\langle \cone(s)\mid s\in S \rangle_\otimes$ generated by the cones of maps in~$S$.
Moreover, the central 2-ring of these categories on the objects of $\mathcal R$ is canonically isomorphic to the localized graded commutative 2-ring $S^{-1}\mathcal R$.
\end{thm}

We see in particular that $S^{-1}_\mathcal T\mathcal T$ inherits a canonical $\otimes$-triangulated structure.
In order to prove the theorem we will need a couple of preliminary results.

\begin{prop}
\label{prop:J_S}
$\mathcal J_S=\{x\in \mathcal T\mid \exists s\in S \textrm{ such that } s\otimes \id_x=0 \}$.
\end{prop}
\begin{proof}
The argument is almost precisely as in \cite{balmer:spec3}*{Proposition 3.7}; we briefly recall it as it is easier and slightly more natural in this context.
Write $\mathcal J'$ for the category on the right hand side.
We have $\mathcal J'\subseteq \mathcal J_\mathcal S$ by \cite{balmer:spec3}*{Proposition 2.14}.
For the other inclusion, note that $S\otimes S\subseteq S$ implies that $\mathcal J'$ is equal to
$\{x\in \mathcal T\mid \exists s\in S \textrm{ and } n\geq 1 \textrm{ such that } s^{\otimes n}\otimes \id_x =0 \}$.
The latter is easily seen to be a thick $\otimes$-ideal of $\mathcal T$, where closure under taking cones is a consequence of \cite{balmer:spec3}*{Lemma 2.11}. 
By \cite{balmer:spec3}*{Proposition 2.13} $\mathcal J'$ contains $\cone(s)$ for all $s\in S$. 
Hence we conclude the other inclusion $\smash{\mathcal J_S = \langle \cone(s) \mid s\in S \rangle_\otimes \subseteq \mathcal J' }$ as well.
\end{proof}

\begin{cor}
\label{cor:cone_in_J}
Let $a\colon x\to y$ be any map of $\mathcal T$. Then $\cone(a)$ belongs to $ \mathcal J_S$ if and only if there exist a map $s\in S$ and two maps $b$ and~$c$ as in the following square
\[
\xymatrix{
g\otimes x \ar[r]^-{\id_g \otimes a} \ar[d]_{s\otimes \id_x}
 & g\otimes y \ar[d]^{s\otimes \id_y}  \ar@<-.5ex>[dl]_b \ar@<.5ex>[dl]^c \\
h\otimes x \ar[r]_-{\id_h\otimes a} & h\otimes y
}
\]
such that $b(\id_g\otimes a) = s\otimes \id_x$ and $(\id_h\otimes a)c=s\otimes \id_y$.
\end{cor}

\begin{proof}
This proof is the same as \cite{balmer:spec3}*{Lemma 3.8}, using Proposition \ref{prop:J_S} instead of \cite{balmer:spec3}*{Proposition 3.7}.
\end{proof}

\begin{proof}[Proof of Theorem \ref{thm:central_loc}]
The last assertion in the theorem follows immediately from Proposition \ref{prop:comparison_mult_sys}, with $F$ the fully faithful inclusion $\mathcal R\hookrightarrow \mathcal T$. 
To see the isomorphism of categories, note that for each $s\in S$ we have $\cone (s)\in \mathcal J_S$ by definition, hence the universal property of $\loc\colon \mathcal T\to S^{-1}_\mathcal T\mathcal T$ induces a unique functor $\tilde q\colon S^{-1}_\mathcal T\mathcal T\to \mathcal T/\mathcal J_S$ which is the identity on objects. 
We must show that $\tilde q$ is full and faithful.
Let 
\[
\xymatrix{
x \ar[r]^-a & z & \ar[l]_-t^-\sim y
}
\]
be a fraction in $\mathcal T$ representing a morphism $x\to y$ in~$\mathcal T/\mathcal J_S$. 
Thus $\cone(t)\in \mathcal J_S$ by construction, and by Corollary \ref{cor:cone_in_J} there exist a map $s\colon g\to h$ in~$S$ and some map $b\colon g\otimes z \to h \otimes y$ with $b(g \otimes t)= s\otimes y$ (we will not need the second map~$c$).
Build the following commutative diagram in~$\mathcal T$.
\[
\xymatrix{
x \ar[r]^-a \ar@{..>}@/_6ex/[ddrr]_{a'\;:= } &
 z &
  y \ar[l]_-t^-\sim \ar@{..>}@/^8ex/[dd]^{=:\; t'} \\
&
 g^\vee g z \ar[u]^\simeq \ar[dr]_{g^\vee b}
  & g^\vee g y \ar[u]_\simeq \ar[l]_-{g^\vee g t} \ar[d]^{g^\vee s y} \\
&& g^\vee hy
}
\]
Defining $a'$ and $t'$ as pictured, we see that $t^{-1}a=t'^{-1}a'$ in $\mathcal T/\mathcal J_S$.
But $t'\in S_\mathcal T$, and therefore $t'^{-1}a'$ lies in the image of~$\tilde q$. 
This shows that $\tilde q$ is full.
To prove that it is faithful, consider a fraction
\[
\varphi \;\colon \;
\xymatrix{
x \ar[r]^-a & w & \ar[l]_-{s'}^-\sim y
}
\]
representing a morphism $\varphi\colon x\to y$ in $S^{-1}_\mathcal T\mathcal T$ (thus $s'\in S_\mathcal T$), and assume that $\tilde q(s'^{-1}a)=0$. This means that there exists a commutative diagram
\begin{equation*}
\label{eq:other_zeros}
\xymatrix@R=5pt{
& w \ar[dd] & \\
x \ar[ur]^a \ar[dr]_0 && y \ar[ul]_{s'} \ar[dl]^t \\
& z &
}
\end{equation*}
with $\cone(t)\in \mathcal J_S$. Applying again Corollary \ref{cor:cone_in_J} to~$t$ we obtain $s \colon g\to h$ in $S$ and $b\colon gz\to hy$ such that $b(g\otimes t)=s\otimes y$ (precisely as above), and we may construct a commutative diagram as follows:
\begin{equation}
\label{eq:faithful}
\xymatrix{
& w \ar[d]  \ar@{..>}@/_12ex/[dddr]_{d \; := } & \\
x \ar[ur]^a \ar[r]^-0  &  
 z &
  y \ar[l]_-t \ar[ul]_{s'} \ar@{..>}@/^8ex/[dd]^{=:\; s''} \\
&
 g^\vee g z \ar[u]^\simeq \ar[dr]_{g^\vee b}
  & g^\vee g y \ar[u]^\simeq \ar[l]_-{g^\vee g t} \ar[d]^{g^\vee s y}  \\
&& g^\vee hy 
}
\end{equation}
Note that $s''$, as defined in the diagram, belongs to~$S_\mathcal T$.
Thus setting $d$ as indicated we may deduce from \eqref{eq:faithful} the existence of a commutative diagram
\begin{equation*}
\xymatrix@R=5pt{
& w \ar[dd]_d & \\
x \ar[ur]^a \ar[dr]_0 && y \ar[ul]_{s'} \ar[dl]^{s''} \\
& z &
}
\end{equation*}
in $\mathcal T$, showing that $\varphi=0$. Hence $\tilde q$ is faithful, thus completing the proof of Theorem~\ref{thm:central_loc}.
\end{proof}

\subsection{Generalized comparison maps}

We now extend the definition and the basic properties of Paul Balmer's comparison map~$\rho$ from triangular to Zariski spectra.

As before, let $\mathcal T$ be an essentially small tensor triangulated category.

\begin{thm}
\label{thm:rho}
For every central 2-ring $\mathcal R$ of $\mathcal T$ there is a continuous spectral map 
$\rho^\mathcal R_\mathcal T\colon \Spc \mathcal T\to \Spec \mathcal R$ 
which sends the prime thick $\otimes$-ideal $\mathcal P\subset \mathcal T$ to the prime ideal
\[
\rho^\mathcal R_\mathcal T (\mathcal P):=\{ r \in \Mor \mathcal R \mid \cone(r) \not\in \mathcal P \} \, .
\]
Moreover, the map $\rho^\mathcal R_\mathcal T$ is natural in the following sense: if $F\colon \mathcal T\to \mathcal T'$ is a tensor-exact functor and $\mathcal R'$ is a central 2-ring of $\mathcal T'$ such that $F\mathcal R\subseteq \mathcal R'$, then the square of spectral continuous maps
\[
\xymatrix{
\Spc \mathcal T' \ar[d]_{\rho^{\mathcal R'}_{\mathcal T'}} \ar[r]^-{\Spc F} &
  \Spc \mathcal T \ar[d]^{\rho^{\mathcal R}_{\mathcal T}} \\
\Spec \mathcal R' \ar[r]^-{\Spec F} &
 \Spec \mathcal R
}
\]
is commutative.
\end{thm}

\begin{proof}
Let $\mathcal P\in \Spc \mathcal T$ and denote by $q_\mathcal P\colon \mathcal T\to \mathcal T/\mathcal P$ the Verdier quotient functor. 
The functor $q_\mathcal P$ is strong monoidal so the full subcategory $\mathcal R_\mathcal P:= \{ q_\mathcal P(g) \mid g\in \mathcal R\} \subseteq \mathcal T/\mathcal P$ is a central 2-ring of $\mathcal T/\mathcal P$. 
Since $\mathcal T/\mathcal P$ is a local tensor triangulated category its central 2-ring $\mathcal R_\mathcal P$ has a unique maximal ideal $\mathfrak m_\mathcal P$ consisting the non-invertible maps by Theorem \ref{thm:local_Rtot}.
By thickness of $\mathcal P$ we have $\mathcal P = \Ker (q_\mathcal P)$ and therefore an equality of sets
\[
\rho^\mathcal R_\mathcal T(\mathcal P)= q_\mathcal P^{-1}(\mathfrak m_\mathcal P) \; \subseteq \; \Mor \mathcal R \,.
\]
Thus $\rho^\mathcal R_\mathcal T(\mathcal P)$ is the preimage of the unique maximal ideal of $\mathcal R_\mathcal P$ under the morphism  $q_\mathcal P\colon \mathcal R\to \mathcal R_\mathcal P$ of graded commutative 2-rings, and in particular it is a homogeneous prime. 
This shows that the resulting function $\rho^\mathcal R_\mathcal T \colon \Spc \mathcal T\to \Spec \mathcal R$ is well-defined.
To see that it is a spectral continuous map it suffices to note that, by definition,
\[
(\rho^\mathcal R_\mathcal T)^{-1}(D_r) = U(\cone(r))
\]
for every $r\in \mathcal R$, where $D_r=\{ \mathfrak p \mid r\not\in \mathfrak p \}$ is a quasi-compact basic open for the Zariski topology of $\Spec \mathcal R$, and $U(x)=\{\mathcal P\mid x\in \mathcal P \}$ (for $x\in \mathcal T$) is a quasi-compact basic open for the Zariski topology of $\Spc \mathcal T$. 

The naturality of $\rho^\mathcal R_\mathcal T$ in the pair $(\mathcal T,\mathcal R)$ can be checked immediately from the definitions.
\end{proof}

\subsection{A criterion for injectivity}

As above let $\mathcal{R}$ be a central 2-ring in an essentially small tensor triangulated category~$\mathcal T$, and let $\rho^\mathcal R_\mathcal T$ be the associated continuous map of Theorem~\ref{thm:rho}. 
We have the following topological condition which implies the injectivity of $\rho^\mathcal R_\mathcal T$.

\begin{prop}\label{prop:injectivity}
Suppose the collection of subsets
\begin{displaymath}
\mathcal{B} = \{\supp (\cone(r)) \; \vert \; r \in \Mor \mathcal R\}
\end{displaymath}
gives a basis of closed subsets for the Zariski topology on $\Spc \mathcal{T}$. Then the comparison map $\rho^\mathcal R_\mathcal T$ is injective, and is furthermore a homeomorphism onto its image. 
\end{prop}

\begin{proof}
Suppose first that $\mathcal{B}$ is a basis of closed subsets. Let $\mathcal{P}, \mathcal{Q} \in \Spc{\mathcal{T}}$ be such that $\rho^\mathcal R_\mathcal T(\mathcal{P}) = \rho^\mathcal R_\mathcal T(\mathcal{Q})$, \ie, 
$\cone(r)\notin \mathcal{P}$ if and only if $\cone(r) \notin \mathcal{Q}$ for every $r\in \Mor \mathcal R$. 
Using our basis $\mathcal{B}$  we see that
\begin{displaymath}
\overline{\{\mathcal{P}\}} = \bigcap_{\substack{r\in \Mor \mathcal R \\ \cone(r)\notin \mathcal{P}}} \supp{(\cone(r))} = \bigcap_{\substack{r\in \Mor \mathcal R \\ \cone(r)\notin \mathcal{Q}}} \supp{(\cone(r))} = \overline{\{\mathcal{Q}\}}
\end{displaymath}
where the middle equality follows from $\rho^\mathcal R_\mathcal T(\mathcal{P}) = \rho^\mathcal R_\mathcal T(\mathcal{Q})$. But then $\mathcal{P} = \mathcal{Q}$ since the space $\Spc{\mathcal{T}}$ is~$T_0$, proving that $\rho^\mathcal{R}_\mathcal{T}$ is injective.

Recall that $\Spec \mathcal{R}$ has a basis of open subsets given by the $D_r$ for $r\in \mathcal{R}$. Now observe that, as was already noted in the proof of Theorem \ref{thm:rho}, we have by definition (and injectivity of $\rho^\mathcal{R}_\mathcal{T}$)
\begin{displaymath}
\rho^\mathcal{R}_\mathcal{T}(\Spc \mathcal{T}) \cap D_r = \rho_{\mathcal T}^\mathcal R\, U(\cone(r)).
\end{displaymath}
By hypothesis the $U(\cone(r))$ are a basis of open subsets for $\Spc \mathcal{T}$ and so we see that $\rho^\mathcal{R}_\mathcal{T}$ is a homeomorphism onto its image.
\end{proof}
\section{Applications}
\label{sec:examples}

\subsection{Graded commutative rings} \label{subsec:grcomm}
Let $G$ be an abelian group and let $R$ be an $\epsilon$-commutative $G$-graded ring. Let us denote by $R$-$\GrMod$ the Grothendieck abelian tensor category of graded (left) $R$-modules with degree zero homomorphisms. As usual $\D(R)$ denotes the unbounded derived category of $R$-$\GrMod$ and $\D^\mathrm{perf}(R)$ denotes the compact objects of $\D(R)$. We recall that the compact objects are precisely those complexes quasi-isomorphic to a bounded complex of finitely generated projective $R$-modules. Also, $\D^\perf(R)$ is a rigid tensor triangulated category for the derived tensor product $\otimes=\otimes^\mathbf L_R$ (this involves signs, see \cite{ivo_greg:graded}) with tensor unit~$R$.

In \cite{ivo_greg:graded} we have given a classification of the thick tensor ideals of $\D^\mathrm{perf}(R)$ in the case that $R$ is noetherian. The aim of this section is to demonstrate how to use the generalised comparison map~$\rho$ we have constructed to remove the noetherian hypothesis from this classification of thick tensor ideals. 
The result and the argument are similar in spirit to the work of Thomason \cite{Thomclass} on perfect complexes on quasi-compact and quasi-separated schemes. However, our approach is rather more formal - the main input from graded commutative algebra occurs almost exclusively in the results of \cite{ivo_greg:graded} and what remains here is an abstract argument belonging essentially to the realm of tensor triangular geometry.

Let us write $R \cong \colim_\Lambda R_\lambda$ where $\Lambda$ is a small filtered category and each $R_\lambda$ is a finitely generated $G$-graded subring of $R$ (i.e.\ consider $R$ as the union of all its finitely generated graded subrings). In particular each $R_\lambda$ is a noetherian $\epsilon$-commutative $G$-graded ring and so our classification theorem (\cite{ivo_greg:graded}*{Theorem 5.1}) applies to give homeomorphisms
\begin{displaymath}
\specgr R_\lambda \stackrel{\sim}{\to} \Spc \D^\mathrm{perf}(R_\lambda).
\end{displaymath}
We will denote by $i_\lambda$ the inclusion $R_\lambda \to R$ and by $\LL i^*_\lambda$ the associated functor $\D^\mathrm{perf}(R_\lambda) \to \D^\mathrm{perf}(R)$.

We let $\mathcal{R}$ be the graded commutative 2-ring given by taking the replete closure of the full subcategory $\{R(g) \; \vert \; g\in G\}$ of $\D^\mathrm{perf}(R)$. Similarly $\mathcal{R}_\lambda$ denotes the replete closure of the full subcategory $\{R_\lambda(g) \; \vert \; g\in G\}$ of $\D^\mathrm{perf}(R_\lambda)$. It is clear that $\LL i^*_\lambda \mathcal{R}_\lambda$ is contained in $\mathcal{R}$ and there are similar containments coming from the left derived functors of extension of scalars along the structure maps in the directed system given by $\Lambda$. 

  We denote by $\rho$ (resp.\ $\rho_\lambda$) the comparison map from $\Spc\D^{\mathrm{perf}}(R)$ to $\Spec \mathcal{R}$ (resp.\ $\Spc\D^{\mathrm{perf}}(R_\lambda)$ to $\Spec \mathcal{R}_\lambda$) of Theorem \ref{thm:rho}. We note that due to the naturality of the comparison maps there are commutative squares as follows for each~$\lambda$.
\begin{displaymath}
\xymatrix{
\Spc \D^\mathrm{perf}(R) \ar[rr]^-{\Spc \LL i^*_\lambda} \ar[d]_\rho && \Spc \D^\mathrm{perf}(R_\lambda) \ar[d]^{\rho_\lambda} \\
\Spec \mathcal{R} \ar[rr]_-{\Spec  i^*_\lambda} && \Spec \mathcal{R}_\lambda
}
\end{displaymath}

Next let us make some observations about compatibilities between the homogeneous spectra of $R$, $\mathcal{R}$, and the companion category $\mathcal{C}_R$ (of course all of these observations are equally valid for the $R_\lambda$ and will be used below, together with the obvious compatibility conditions coming from the directed system of subrings). By construction the companion category $\mathcal{C}_R$  is canonically monoidally equivalent to $\mathcal{R}$ (see \cite{ivo_greg:graded}*{Definition 2.1} for details on the companion category and \emph{loc}.\ \emph{cit}.\ Proposition 2.14 concerning the equivalence being symmetric). 
Thus, since (prime) ideals in a graded commutative 2-ring are closed under composition with isomorphisms, the topological spaces $\Spec \mathcal{C}_R$ and $\Spec \mathcal{R}$ are canonically homeomorphic. 

\begin{lemma}
There is a canonical homeomorphism
\begin{displaymath}
\specgr R \stackrel{\sim}{\to} \Spec \mathcal{C}_R.
\end{displaymath}
\end{lemma}

\begin{proof}
Given any homogeneous ideal $I$ of $R$ we can associate to it a unique ideal $\mathcal{I}$ of $\mathcal{C}_R$ by closing the elements of $I$, viewed as morphisms out of the tensor unit $0$ in $\mathcal{C}_R$, under tensoring with objects. On the other hand any ideal $\mathcal{I}$ of $\mathcal{C}_R$ gives an ideal of $R$ consisting of all the morphisms of $\mathcal{C}_R(0,g)$ in $\mathcal{I}$ as $g$ varies over $G$. It is not hard to check that these assignments are inverse to one another.
\end{proof}

Let us also note that $\mathcal{C}_R$ is, essentially by definition, the union of the subcategories $\mathcal{C}_{R_\lambda}$. Combining this with the observations we have already made shows that we may equally well speak of graded rings, companion categories, or the central 2-rings we have defined inside perfect complexes without introducing any ambiguity. We shall switch between these different points of view whenever it is convenient.

Our aim is to prove that the comparison map $\rho$ with target $\Spec \mathcal{R} \cong \specgr R$ is a homeomorphism. The bulk of the proof will be contained in a string of rather simple lemmas and observations. We begin by showing that it is injective by applying the criterion of Proposition \ref{prop:injectivity}.

\begin{notation}
We denote the support on $\D^\mathrm{perf}(R)$, in the sense of Balmer, by $\supp_R$ and the homological (or small) support by $\supph_R$. Strictly speaking $\supph_R$ takes values in $\specgr R$ but we will abuse notation slightly by considering it also to take values in $\Spec \mathcal{R}$ via the canonical identification. We use similar notation for the $R_\lambda$ but note that the distinction for these noetherian rings is not so important due to the following lemma. 
\end{notation}

\begin{lemma}\label{lem:noeth_homeo}
Suppose that $R$ is noetherian. Then, with $\mathcal{R}$ as above, $\rho$ is a support preserving homeomorphism.
\end{lemma}
\begin{proof}
By \cite{ivo_greg:graded}*{Theorem 5.1} the morphism $\sigma\colon \spec \mathcal R\to \Spc \D^\mathrm{perf}(R)$ given by
\begin{displaymath}
\sigma(\mathfrak{p}) = \{E\in \D^\mathrm{perf}(R)\; \vert \; \mathfrak{p}\notin \supph_R E\}
\end{displaymath}
is a homeomorphism identifying $\supp_R$ and $\supph_R$. It is then easily checked that $\rho$ is inverse to $\sigma$. For instance we have, for $\mathfrak{p}\in \Spec \mathcal{R}$
\begin{align*}
\rho\sigma(\mathfrak{p}) &= \rho\{E\in \D^\mathrm{perf}(R)\; \vert \; \mathfrak{p}\notin \supph_R E\} \\
&= \rho\{E\in \D^\mathrm{perf}(R)\; \vert \; \sigma(\mathfrak{p})\notin \supp_R E\}\\
&= \{r\in \mathcal{R} \; \vert \; \sigma(\mathfrak{p})\in \supp_R\cone(r)\} \\
&= \{r\in \mathcal{R} \; \vert \; \mathfrak{p}\in \supph_R\cone(r)\} \\
&= \mathfrak{p} \,.
\end{align*}
\end{proof}

\begin{lemma}\label{lem:supports_agree}
For every object $E$ of $\D^\mathrm{perf}(R)$ there is an equality
\begin{displaymath}
\rho^{-1} \supph_R E = \supp_R E.
\end{displaymath}
\end{lemma}
\begin{proof}
Let $E$ be a perfect complex over $R$ as in the statement of the lemma. We may assume, by choosing an isomorphic object if necessary, that $E$ is in fact a bounded complex of finitely generated free $R$-modules. As $E$ is determined by finitely many matrices with coefficients in $R$ we can find some $\lambda$ and a perfect complex $E_\lambda \in \D^{\mathrm{perf}}(R_\lambda)$ so that
\begin{displaymath}
E \cong R \otimes_{R_\lambda} E_\lambda = \LL i^*_\lambda E_\lambda.
\end{displaymath}
Naturality of the comparison map, together with the canonical isomorphisms we have observed, gives a commutative diagram
\begin{displaymath}
\xymatrix{
\Spc \D^\mathrm{perf}(R) \ar@{..>}@/_6ex/[dd]_{\rho'\, :=}
\ar[rr]^-{\Spc \LL i^*_\lambda} \ar[d]_\rho &&
\ar@{..>}@/^6ex/[dd]^{=:\, \rho'_\lambda}
 \Spc \D^\mathrm{perf}(R_\lambda) \ar[d]^{\rho_\lambda} \\
\Spec \mathcal{R} \ar[rr]_-{\Spec i^*_\lambda} \ar[d]_{\simeq} && \Spec \mathcal{R}_\lambda \ar[d]^{\simeq} \\
\specgr R \ar[rr]_{\specgr i_\lambda} && \specgr R_\lambda
}
\end{displaymath}

Thus we deduce equalities
\begin{align*}
(\rho')^{-1}\supph_R(E) &= (\rho')^{-1}(\specgr i_\lambda)^{-1} \supph_{R_\lambda}(E_\lambda) \\
&= (\Spc \LL i^*_\lambda)^{-1}(\rho_\lambda ')^{-1} \supph_{R_\lambda}(E_\lambda) \\
&= (\Spc \LL i^*_\lambda)^{-1} \supp_{R_\lambda}(E_\lambda) \\
&= \supp_R (\LL i^*_\lambda E_\lambda) \\
&= \supp_R E
\end{align*}
where the first equality is standard, the third equality follows from \cite{ivo_greg:graded}*{Theorem 5.1} which applies as $R_\lambda$ is noetherian, the fourth equality is \cite{balmer:prime}*{Proposition 3.6}, the final equality holds by the definition of $E_\lambda$, and the other two follow from commutativity of the diagram of comparison maps.
\end{proof}

\begin{lemma}\label{lem:graded_basis}
The collection of subsets
\begin{displaymath}
\mathcal{B} = \{\supp_R (\cone(r)) \; \vert \; r \in \Mor \mathcal{R}\}
\end{displaymath}
is a basis for the Zariski topology on $\Spc \D^\mathrm{perf}(R)$. Thus $\rho$ is a homeomorphism onto its image.
\end{lemma}
\begin{proof}
Let $E$ be an object of $\D^\mathrm{perf}(R)$ and let $\mathcal{P}$ be a prime ideal not in $\supp_R E$. We need to show that there exists a map $r$ in $\mathcal{R}$ such that $\supp_R E \subseteq \supp_R \cone(r)$ and $\mathcal{P}\notin \supp_R \cone(r)$. Since the subsets $\supph_R \cone(s)$ as $s$ varies over the morphisms in $\mathcal{R}$ form a basis for the Zariski topology on $\Spec \mathcal{R}$ we can find an~$r$ such that $\supph_R E \subseteq \supph_R \cone(r)$ and $\rho(\mathcal{P})\notin \supph_R \cone(r)$. Applying $\rho^{-1}$ and using the last lemma shows that $\supp_R \cone(r)$ is the desired subset.

The result then follows from Proposition \ref{prop:injectivity}.
\end{proof}

Let us now show that $\rho$ is also surjective. In fact we will show the equivalent statement that the composite
\begin{displaymath}
\rho'\colon \Spc \D^\mathrm{perf}(R) \to \Spec \mathcal{R} \stackrel{\sim}{\to} \specgr R
\end{displaymath}
is surjective.

\begin{lemma}
The comparison map $\rho'$ is surjective.
\end{lemma}
\begin{proof}
Let $\frp$ be a homogeneous prime ideal of $\specgr R$. Then, setting $\frp_\lambda = \frp~\cap~R_\lambda$, we can write $\frp$ as the filtered colimit $\frp = \colim_\Lambda \frp_\lambda$. As each $\frp_\lambda$ is prime in $R_\lambda$ we can find, by Lemma \ref{lem:noeth_homeo}, a unique $\mathcal{P}_\lambda \in \Spc \D^\mathrm{perf}(R)$ such that $\rho'_\lambda (\mathcal{P}_\lambda) = \frp_\lambda$.

We define a full subcategory $\mathcal{P} := \bigcup_\Lambda \LL i^*_\lambda \mathcal{P}_\lambda$ of $\D^\mathrm{perf}(R)$. We claim that $\mathcal{P}$ is a prime $\otimes$-ideal. In order to see this first note that if $R_{\lambda_1} \subseteq R_{\lambda_2}$ then, since $\mathfrak{p}_{\lambda_1} = R_{\lambda_1}\cap \mathfrak{p}_{\lambda_2}$, if $r\in R_{\lambda_1}$ is not in $\mathfrak{p}_{\lambda_1}$ it is not in $\mathfrak{p}_{\lambda_2}$. Thus, letting $j\colon R_{\lambda_1} \to R_{\lambda_2}$ denote the inclusion, this observation together with the fact that the derived pullback sends cones to cones yields
\begin{displaymath}
\LL j^*\mathcal{P}_{\lambda_1} \subseteq \langle \LL j^* \cone(r)\; \vert \; r\in (R_{\lambda_1}\smallsetminus \mathfrak{p}_{\lambda_1})\rangle_\otimes \subseteq \mathcal{P}_{\lambda_2}. 
\end{displaymath}
So $\mathcal{P}$ is an increasing filtered union and is thus a prime $\otimes$-ideal as any perfect complex over $R$ can be obtained from a perfect complex over some $R_\lambda$ via $\LL i^*_\lambda$.

We now show that $\rho' (\mathcal{P}) = \frp$. Let $r$ be a homogeneous element of $R$ and let $\lambda \in \Lambda$ be such that $r\in R_\lambda$. Then~$r$ lies in $\frp$ if and only if $r$ lies in $\frp_\lambda$,
 if and only if $\cone_{R_\lambda}(r)$ is not in $\mathcal{P}_\lambda$,
  if and only if $\LL i^*_\lambda \cone_{R_\lambda}(r) \cong \cone_R(r)$ is not in~$\mathcal{P}$. 
  Thus $\rho' (\mathcal{P}) = \frp$ and we see that $\rho'$ is surjective as claimed.
\end{proof}

Combining the previous lemmas, we obtain the following theorem. 

\begin{thm}
\label{thm:grcommrings}
Let $R$ be a $G$-graded $\epsilon$-commutative ring for some abelian group $G$ and some bilinear form $\epsilon\colon G\times G\to \Z/2$. 
Then there is a (unique) homeomorphism
\begin{displaymath}
\Spc \D^\mathrm{perf}(R) \cong \Spec \mathcal{R} \cong \specgr R
\end{displaymath}
which identifies the support in the sense of Balmer with the usual homological support.
\end{thm}

\subsection{Schemes with an ample family of line bundles}
Throughout this section all schemes $X$ are quasi-compact and quasi-separated. This is the necessary and sufficient hypothesis for Balmer's reconstruction theorem $X\cong \Spc \D^\perf(X)$ to apply (see \cite{balmer:icm}*{Theorem 54}).  
 We will show that every ample family of line bundles on $X$ gives rise to an injective comparison map from the spectrum of $\D^\mathrm{perf}(X)$ to that of a symmetric 2-ring associated with the family. Let us begin by recalling what it means for a collection of line bundles on  $X$ to be ample.

\begin{defi}
Let $\{\mathcal{L}_{\lambda}\}_{\lambda\in \Lambda}$ be a non-empty collection of line bundles on $X$. We say that  $\{\mathcal{L}_{\lambda}\}_{\lambda\in \Lambda}$ is an \emph{ample family of line bundles} if there is a family of sections $f\in H^0(X,\mathcal{L}_\lambda^n)$ for $\lambda \in \Lambda$ and $n\geq 0$ such that the open sets
\begin{displaymath}
X_f = \{x\in X \; \vert \; f_x\notin \mathfrak{m}_x(\mathcal{L}_\lambda^n)_x\}
\end{displaymath}
form a basis for~$X$. Here of course $\mathcal L^n_\lambda$ denotes the $n$th tensor power $\mathcal (\mathcal L_\lambda)^{\otimes n}$.
\end{defi}

Given a section $f$ of a line bundle $\mathcal{L}$ on $X$ we shall denote by $Z(f)$ the closed complement of $X_f$.

\begin{lemma}\label{lemma:bundlesupport}
Let $X$ be a scheme, $\mathcal{L}$ a line bundle on $X$, and $f$ a section of $\mathcal{L}$. Then, via the homeomorphism $X\to \Spc \D^\mathrm{perf}(X)$, we have
\begin{displaymath}
\supp \cone(\mathcal{O}_X \stackrel{f}{\to} \mathcal{L}) = Z(f).
\end{displaymath}
\end{lemma}
\begin{proof}
This is essentially immediate as the homeomorphism $X\cong \Spc \D^\mathrm{perf}(X)$ of \cite{balmer:icm}*{Theorem 54} identifies the support in the sense of Balmer with the homological support.
\end{proof}

\begin{defi}
Let $X$ be a scheme and suppose that we are given some non-empty collection $\underline{\mathcal{L}} = \{\mathcal{L}_{\lambda}\}_{\lambda\in \Lambda}$ of line bundles on $X$. We denote by $\mathcal{R}(\underline{\mathcal{L}})$ the associated central 2-ring of $\D^\mathrm{perf}(X)$, i.e., the replete closure  of the full subcategory of $\D^\mathrm{perf}(X)$ whose objects are 
\begin{displaymath}
\{\mathcal{L}_{\lambda_1}^{m_1} \otimes \cdots \otimes \mathcal{L}_{\lambda_n}^{m_n}\; \vert \; n\geq 1, \lambda_i\in \Lambda, m_i\in \mathbb{Z}\}.
\end{displaymath}
\end{defi}

\begin{thm}
\label{thm:ample}
Let $X$ be a quasi-compact and quasi-separated scheme with an ample family of line bundles $\underline{\mathcal{L}} = \{\mathcal{L}_\lambda\}_{\lambda \in \Lambda}$. Then the comparison map 
\[
\rho \colon \Spc \D^\mathrm{perf}(X) \to \Spec \mathcal{R}(\underline{\mathcal{L}})
\]
is a homeomorphism onto its image. In particular, there is an injective morphism $\rho_X^{\underline{\mathcal L}} \colon X\to \Spec\mathcal{R}(\underline{\mathcal{L}})$ and $X$ has the subspace topology relative to this injection.
\end{thm}
\begin{proof}
By Proposition \ref{prop:injectivity} it is enough to check that the supports of cones on morphisms in $\mathcal{R}(\underline{\mathcal{L}})$ give a basis of closed subsets for the Zariski topology on $\Spc \D^\mathrm{perf}(X) \cong X$. From Lemma \ref{lemma:bundlesupport} we know that, by taking the support of the cone of the map associated to a global section $f$ of $\mathcal{L}_\lambda^n$, we obtain the subset corresponding to $Z(f)$. Since the family of line bundles is ample these subsets form a basis of closed subsets and so we are done.
\end{proof}

\begin{remark}\label{rem:ffs}
It is natural to compare this result with work of Brenner and Schr\"{o}er \cite{brenner-schroer}. They define a notion of $\Proj$ for multihomogeneous rings, and in their Theorem~4.4 characterise ampleness of a family of line bundles $\mathcal{L}_1,\ldots,\mathcal{L}_n$ on a quasi-compact and quasi-separated scheme $X$ in terms of the canonical rational map
\begin{displaymath}
\xymatrix{
X \ar@{-->}[r] & \Proj \Gamma(X, \bigoplus_{d\in \mathbb{N}^n} \mathcal{L}_1^{d_1} \otimes \cdots \otimes \mathcal{L}_n^{d_n}).
}
\end{displaymath}
In the case that $X$ is divisorial one can verify that, identifying $\Spc \D^\mathrm{perf}(X)$ with $X$, the map of Theorem~\ref{thm:ample} can be used to recover the morphism of Brenner and Schr\"{o}er at the level of topological spaces.
\end{remark}

\begin{example}[Projective schemes]
\label{ex:projective}
Consider $X= \Proj(R)$, where $R$ is a commutative non-negatively $\Z$-graded $k$-algebra over a field~$k$, such that $R_0=k$ and such that $R_1$ generates $R$ over~$k$. 
Then $\mathcal L := \mathcal O(1)$ is an ample line bundle, so Theorem \ref{thm:ample} yields a topological embedding $\rho^{\mathcal L}_X\colon X\hookrightarrow \spec \mathcal R(\mathcal L)$.
Now it is not hard to see that we have an isomorphism $\spec \mathcal R(\mathcal L)\cong \specgr R$ which identifies $\rho^{\mathcal L}_X$ with the defining embedding of $X$ into $\specgr R$ (and therefore also with Balmer's graded comparison map $\rho^{*,\mathcal O(1)}\colon X=\Spc \D^\perf(X)\to \specgr \End^{*,\mathcal O(1)}(\unit)$, \cf\ \cite{balmer:spec3}*{Remark 8.2}).

Since $\specgr R$ consists of $X$ plus a unique closed point, this example shows that, in general, the map $\rho^{\underline{\mathcal L}}_X$ need not be surjective, nor must it have closed image.
\end{example}

\begin{example}[Affine schemes with torsion Picard group]
Let $R$ be a ring such that the Picard group, $\Pic(R)$, is torsion. Denote by $\mathcal{R}\subseteq \D^\mathrm{perf}(R)$ the central 2-ring consisting of all line bundles in degree zero. In this case, as in the case where one considers Balmer's $\rho$, the map $\rho_{\D^\mathrm{perf}(R)}^\mathcal{R}$ is a homeomorphism. Let us sketch the argument.

Consider the composite
\begin{displaymath}
f:=(\Spec R \stackrel{\sim}{\to} \Spc \D^\mathrm{perf}(R) \to \Spec \mathcal{R})
\end{displaymath}
which, unwinding the definitions, sends $\mathfrak{p}\in \Spec R$ to the prime ideal $\{s\in \mathcal{R} \; \vert \; \mathfrak{p} \in \supph \cone(s)\}$. There is also a map $g\colon \Spec \mathcal{R} \to \Spec R$ sending a prime ideal $P$ to $P(R,R)$. It is clear that $gf = \id_{\Spec R}$.

Now suppose that $P\in \Spec \mathcal{R}$ and consider $fg(P)$. An element $s\in \mathcal{R}(\mathcal{L},\mathcal{M})$ is in a given ideal if and only if the translate $\tilde{s}\colon R\to \mathcal{L}^{-1}\otimes \mathcal{M}$ is. Let $n$ be the order of $\mathcal{L}^{-1}\otimes \mathcal{M}$ and consider a composite
\begin{displaymath}
t:=(R \stackrel{\sim}{\to} R^{\otimes n} \stackrel{\tilde{s}^{\otimes n}}{\to} (\mathcal{L}^{-1}\otimes \mathcal{M})^{\otimes n} \stackrel{\sim}{\to} R).
\end{displaymath}
Regardless of the choice of isomorphisms we note that $t$ lies in a given prime ideal if and only if $\tilde{s}$ does and hence if and only if $s$ does. The final observation is that the homological supports of the cones on $s$, $\tilde{s}$, $\tilde{s}^{\otimes n}$, and $t$ agree. Thus $s\in fg(P)$ iff $P(R,R)\in \supph\cone(s)$ iff $P(R,R)\in \supph\cone(t)$ iff $t\in P(R,R)$ iff $s\in P$. Hence $fg$ is the identity on $\Spec \mathcal{R}$ so we have the claimed homeomorphism.
\end{example}

\end{document}